\documentclass[12pt]{article}

\usepackage{amsmath,amsthm,amssymb}
\usepackage[colorlinks,
            linkcolor=red,
            anchorcolor=blue,
            citecolor=magenta
            ]{hyperref}
\usepackage{pdfsync}
\usepackage[numbers,sort&compress]{natbib}
\allowdisplaybreaks

\theoremstyle{definition}

\theoremstyle{plain}
\newtheorem{theorem}{Theorem}[section]
\newtheorem{definition}{Definition}[section]
\newtheorem{remark}{Remark}[section]
\newtheorem{lemma}{Lemma}[section]

\newtheorem{proposition}{Proposition}[section]

\numberwithin{equation}{section}

\newcommand{\vs}{\vspace}

\begin{document}

\title{Normalized solutions and stability for biharmonic Schr\"odinger equation with potential on waveguide manifold\footnote{  This work was partially supported by NNSFC (No. 12171493).}}

\author{ Jun Wang$^{a}$, Zhaoyang Yin$^{a, b}$\footnote {Corresponding author. wangj937@mail2.sysu.edu.cn (J. Wang), mcsyzy@mail.sysu.edu.cn (Z. Yin)
} \\
{\small $^{a}$Department of Mathematics, Sun Yat-sen University, Guangzhou, 510275, China } \\
{\small $^{b}$School of Science, Shenzhen Campus of Sun Yat-sen University, Shenzhen, 518107, China } \\
}

	\date{}

	\maketitle

\date{}

 \maketitle \vs{-.7cm}

  \begin{abstract}
In this paper, we
study the  following biharmonic Schr\"odinger equation with potential and mixed nonlinearities
\begin{equation*}
 \left\{\aligned
&\Delta^2 u +V(x,y)u+\lambda u =\mu|u|^{p-2}u+|u|^{q-2}u,\ (x, y) \in \Omega_r \times \mathbb{T}^n, \\
&\int_{\Omega_r\times\mathbb{T}^n}u^2dxdy=\Theta,
\endaligned
\right.
\end{equation*}
where $\Omega_r \subset \mathbb{R}^d$ is an open bounded convex domain, $r>0$ is large and $\mu\in\mathbb{R}$. The exponents satisfy $2<p<2+\frac{8}{d+n}<q<4^*=\frac{2(d+n)}{d+n-4}$, so that the nonlinearity is a combination of a mass subcritical and a mass supercritical term. Under some assumptions on $V(x,y)$ and $\mu$, we obtain the several existence results on waveguide manifold. Moreover, we also consider the orbital stability of the solution.
\end{abstract}

{\footnotesize {\bf   Keywords:}  Biharmonic Schr\"odinger equation; Normalized solutions;  Variational methods; Waveguide manifold.

}

\section{ Introduction and main results}

This paper studies the existence of normalized solutions for the following biharmonic Schr\"odinger equation with potential and mixed nonlinearities
\begin{equation} \label{eq1.1}
 \left\{\aligned
&\Delta^2 u +V(x,y)u +\lambda u=\mu|u|^{p-2}u+|u|^{q-2}u,\ (x, y) \in \Omega_r \times \mathbb{T}^n, \\
&\int_{\Omega_r\times\mathbb{T}^n}u^2dxdy=\Theta,\ u\in H_0^2(\Omega_r\times\mathbb{T}^n),
\endaligned
\right.
\end{equation}
where $\Omega_r \subset \mathbb{R}^d$ is an open bounded convex domain, $r>0$ is large, $d \geq 5$, $2+\frac{8}{d+n}< q<4^*$, the mass $\Theta>0$ and the parameter $\mu \in \mathbb{R}$ are prescribed. The frequency $\lambda$ is unknown and to be determined. The energy functional $I_r: H_0^2(\Omega_r\times\mathbb{T}^n) \rightarrow \mathbb{R}$ is defined by
\begin{equation*}
I_r(u)=\frac{1}{2} \int_{\Omega_r\times\mathbb{T}^n}[|\Delta u|^2+V(x,y)u^2] d xdy -\frac{1}{q} \int_{\Omega_r\times\mathbb{T}^n}|u|^q d xdy-\frac{\mu}{p} \int_{\Omega_r\times\mathbb{T}^n}|u|^p d xdy
\end{equation*}
and the mass constraint manifold is defined by
\begin{equation*}
S_{r, \Theta}=\left\{u \in H_0^2(\Omega_r\times\mathbb{T}^n):\|u\|_2^2=\Theta\right\}.
\end{equation*}
If $\Omega=\mathbb{R}^d$, the energy functional $I: H^2(\mathbb{R}^d\times\mathbb{T}^n) \rightarrow \mathbb{R}$ is defined by
\begin{equation*}
I(u)=\frac{1}{2} \int_{\mathbb{R}^d\times\mathbb{T}^n}[|\Delta u|^2+V(x,y)u^2] d xdy -\frac{1}{q} \int_{\mathbb{R}^d\times\mathbb{T}^n}|u|^q d xdy-\frac{\mu}{p} \int_{\mathbb{R}^d\times\mathbb{T}^n}|u|^p d xdy
\end{equation*}
and the mass constraint manifold is defined by
\begin{equation*}
S_{\Theta}=\left\{u \in H_0^2(\mathbb{R}^d\times\mathbb{T}^n):\|u\|_2^2=\Theta\right\}.
\end{equation*}

Recently, there has been an increasing interest in studying dispersive equations on the waveguide manifolds $\mathbb{R}^d\times\mathbb{T}^n$ whose mixed type geometric nature makes the underlying analysis rather challenging and delicate. For the nonlinear Schr\"odinger model, there have been many results on waveguide manifolds, as shown in \cite{{XCZG2020},{ZHBP2014},{RKJM2021},{YL2023},{ZZH2021},{ZZJZ2021},{NTNV2012},{ADI2012}}. However, there is relatively little research on the biharmonic Schr\"odinger equation on the waveguide manifolds. As far as we know, \cite{XYHY2024} is the first well-posedness and scattering result for biharmonic Schr\"odinger equation without mixed dispersion. After that, Hajaiej et. al in \cite{HHYL2024} consider the following question
\begin{equation*}
  \Delta_{x,y}^2u-\beta \Delta_{x,y}u + \theta u = |u|^{\alpha}u, \ (x,y)\in\mathbb{R}^d\times\mathbb{T}^n ,
\end{equation*}
where $\beta \in \mathbb{R}$ and $\alpha \in (0, \frac{8}{d+n})$. The purpose of this paper is
to study the existence and stability results for normalized ground states. However, they only considered the case where the nonlinear term is $L^2$-subcritical, in which case classical concentration compactness arguments can be used to find the ground state solution. In this article, we consider a more complex scenario where the nonlinear term consists of $L^2$-subcritical and $L^2$-supercritical and has an abstract potential function.  The purpose of this paper is study the following two questions:
\begin{itemize}
  \item the existence of normalized solutions for \eqref{eq1.1};
  \item the orbital stability of the normalized solutions.
\end{itemize}
For the first goal, we use monotonicity techniques. In order to obtain orbital stability of the normalized solutions, we need to modify the classical Cazenave-Lions argument, see \cite{{TCPLL1982}}. However, given the emergence of abstract potential functions, the globally well-posed result for \eqref{eq6.1} is still open, so we need to assume that \eqref{eq6.1} is globally well-posed.

For some results, we expect that $V$ is $C^1$ and consider the function
$$
\widetilde{V}: \mathbb{R}^d\times\mathbb{T}^n \rightarrow \mathbb{R}, \quad \widetilde{V}(z)=\nabla V(z) \cdot z
$$
For $\Omega \subset \mathbb{R}^d$ and $r>0$, let
$$
\Omega_r=\left\{r x \in \mathbb{R}^d: x \in \Omega\right\}
$$
and
$$
S_{r, \Theta}:=S_\Theta \cap H_0^2(\Omega_r\times\mathbb{T}^n)=\left\{u \in H_0^2(\Omega_r\times\mathbb{T}^n):\|u\|_{L^2(\Omega_r\times\mathbb{T}^n)}^2=\Theta\right\} .
$$
From now on we assume that $\Omega \subset \mathbb{R}^d$ is a bounded smooth convex domain with $0 \in \Omega$.

Our assumptions on $V$ are:
\begin{itemize}
\item[$(V_0)$]\ $ V \in C^1(\mathbb{R}^d\times\mathbb{T}^n) \cap L^{\frac{d+n}{4}}(\mathbb{R}^d\times\mathbb{T}^n)$ is bounded and $\left\|V_{-}\right\|_{\frac{d+n}{4}}<S$.

\item[$(V_1)$]\ $V$ is of class $C^1, \lim\limits_{|z| \rightarrow \infty} V(z)=0$, and there exists $\rho \in(0,1)$ such that
$$
\liminf\limits_{|z| \rightarrow \infty} \inf _{z \in B(z_1, \rho|z_1|)}(z_1 \cdot \nabla V(z)) e^{\tau|z_1|}>0 \quad \text { for any } \tau>0.
$$
\end{itemize}
\begin{remark}\label{R1.1}
In order to obtain the existence of normalized solutions in $\mathbb{R}^d\times\mathbb{T}^n$ by taking $\Omega=B_1$, the unit ball centered at the origin in $\mathbb{R}^d$, and analyzing the compactness of the solutions $u_{r, \Theta}$ established in Theorems \ref{t1.1}, \ref{t1.2} and \ref{t1.3} as $r$ tends to infinity, we require the condition $(V_1)$.
\end{remark}

To obtain results, we need an inhomogeneous Gagliardo--Nirenberg inequality on $\mathbb{R}^d \times \mathbb{T}^n$.

\begin{lemma}\label{L2.1}(Gagliardo-Nirenberg inequality). On $\mathbb{R}^d \times \mathbb{T}^n$ we have the Gagliardo-Nirenberg inequality
$$
\|u\|_{q}^q \leq C_{d,n,q}\|\Delta u\|_{2}^{\frac{(d+n)(q-2)}{4}}\|u\|_2^{q-\frac{(d+n)(q-2)}{4}},
$$
where $d+n>4$.
\end{lemma}

The main results of this paper are as follows.
\begin{theorem}\label{t1.1}(\textbf{case $\mu \leq 0$})
Assume $V$ satisfies $\left(V_0\right)$, is of class $C^1$ and $\widetilde{V}$ is bounded, $f$ satisfies $(f_1)-(f_2)$. There hold:

(i) For every $\Theta>0$, there exists $r_\Theta>0$ such that \eqref{eq1.1} on $\Omega_r\times\mathbb{T}^n$ with $r>r_\Theta$ has a mountain pass type solution $\left(\lambda_{r, \Theta}, u_{r, \Theta}\right)$ in $\Omega_r\times\mathbb{T}^n$. Moreover, there exists $C_\Theta>0$ such that
$$
\limsup _{r \rightarrow \infty} \max _{z \in \Omega_r\times\mathbb{T}^n} u_{r, \Theta}(z)<C_\Theta .
$$

(ii) If in addition $\|\widetilde{V}_{+}\|_{\frac{d+n}{4}}<2 S$, then there exists $\widetilde{\Theta}>0$ such that
$$
\liminf _{r \rightarrow \infty} \lambda_{r, \Theta}>0 \text { for any } 0<\Theta<\widetilde{\Theta} .
$$
\end{theorem}

\begin{theorem}\label{t1.2}(\textbf{case $\mu > 0$})
Assume $V$ satisfies $\left(V_0\right)$, $f$ satisfies $(f_1)-(f_2)$ and set
$$
\Theta_V=\left[\frac{1-\left\|V_{-}\right\|_{\frac{d+n}{4}} S^{-1}}{2(d+n)(q-p)}\right]^{\frac{d+n}{4}}\left[\frac{q(8-(d+n)(p-2))}{C_{d,n, q}}\right]^{\frac{8-(d+n)(p-2)}{4 (q-p)}}\left[\frac{(d+n)(q-2)-4}{\mu C_{d,n, p}}\right]^{\frac{(d+n)(q-2)-8}{4 (q-p)}}.
$$
Then the following hold for $0<\Theta<\Theta_V$:

(i) There exists $r_\Theta>0$ such that \eqref{eq1.1} on $\Omega_r\times\mathbb{T}^n$ with $r>r_\Theta$ has a local minimum type solution $\left(\lambda_{r, \Theta}, u_{r, \Theta}\right)$ in $\Omega_r\times\mathbb{T}^n$.

(ii) There exists $C_\Theta>0$ such that
$$
\limsup\limits_{r \rightarrow \infty} \max\limits_{z \in \Omega_r\times\mathbb{T}^n} u_{r, \Theta}(z)<C_\Theta, \quad \liminf\limits_{r \rightarrow \infty} \lambda_{r, \Theta}>0 .
$$
\end{theorem}

\begin{theorem}\label{t1.3}(\textbf{case $\mu > 0$})
Assume $V$ satisfies $\left(V_0\right)$, is of class $C^1$ and $\widetilde{V}$ is bounded, $f$ satisfies $(f_1)-(f_2)$. Set
$$
\widetilde{\Theta}_V=\left(\frac{1-\left\|V_{-}\right\|_{\frac{d+n}{4}} S^{-1}}{2}\right)^{\frac{d+n}{4}}\left(\frac{C_{d,n,q}}{q}A_{p,q}+\frac{C_{d,n, q}}{q}\right)^{-\frac{d+n}{4}}\left(\frac{\mu q C_{d,n, p}}{C_{d,n, q}A_{p,q}}\right)^{\frac{8-(d+n)(q-2)}{4(d+n)(q-p)}},
$$
where
$$
A_{p, q}=\frac{(q-2)((d+n)(q-2)-8)}{(p-2)(8-(d+n)(p-2))}.
$$
Then the following hold for $0<\Theta<\widetilde{\Theta}_V$:

(i) There exists $\widetilde{r}_\Theta>0$ such that \eqref{eq1.1} in $\Omega_r$ admits for $r>r_\Theta$ a mountain pass type solution $\left(\lambda_{r, \Theta}, u_{r, \Theta}\right)$ in $\Omega_r\times\mathbb{T}^n$. Moreover, there exists $C_\Theta>0$ such that
$$
\limsup\limits_{r \rightarrow \infty} \max\limits_{z \in \Omega_r\times\mathbb{T}^n} u_{r, \Theta}(z)<C_\Theta .
$$

(ii) There exists $0<\bar{\Theta} \leq \widetilde{\Theta}_V$ such that
$$
\liminf\limits_{r \rightarrow \infty} \lambda_{r, \Theta}>0 \quad \text { for any } 0<\Theta \leq \bar{\Theta}.
$$
\end{theorem}

If $\Omega=\mathbb{R}^d$, $(V_1)$ is significant for obtaining the following results.

\begin{theorem}\label{t1.4}(\textbf{case $\mu > 0$})
Assume $V$ satisfies $(V_0)-(V_1)$. Then problem \eqref{eq1.1} with $\Omega=\mathbb{R}^d$ admits for any $0<\Theta<\Theta_V$, where $\Theta_V$ is as in Theorem \ref{t1.2}, a solution $\left(\lambda_\Theta, u_\Theta\right)$ with $\lambda_\Theta>0$.
\end{theorem}

\begin{theorem}\label{t1.5}(\textbf{case $\mu > 0$})
Assume $V$ satisfies $(V_0)-(V_1)$. Then \eqref{eq1.1} with $\Omega=\mathbb{R}^d$ admits for $0<\Theta<\bar{\Theta}, \bar{\Theta}>0$ as in Theorem \ref{t1.3} (ii), a solution $(\lambda_\Theta, u_\Theta)$ with $\lambda_\Theta>0$. Moreover, $\lim\limits_{\Theta \rightarrow 0} I(u_\Theta)=\infty$.
\end{theorem}

\begin{theorem}\label{t1.6}(\textbf{case $\mu \leq 0$})
Assume $V$ satisfies $(V_0)-(V_1)$, and $\|\widetilde{V}_{+}\|_{\frac{d+n}{4}}<2 S$. Then problem \eqref{eq1.1} with $\Omega=\mathbb{R}^d$ admits for $0<\Theta<\widetilde{\Theta}, \widetilde{\Theta}>0$ as in Theorem \ref{t1.1}, a solution $(\lambda_\Theta, u_\Theta)$ with $\lambda_\Theta>0$. Moreover, $\lim\limits_{\Theta\rightarrow 0} I\left(u_\Theta\right)=\infty$.
\end{theorem}

\begin{remark}\label{R1.3}
The proof of Theorems \ref{t1.4}-\ref{t1.6} is similar to \cite{TBAQ2023}, so we omit it in this paper.
\end{remark}

\begin{theorem}\label{t1.7}
The solutions obtained in Theorems \ref{t1.2} and \ref{t1.4} are orbitally stable in some
sense.
\end{theorem}

The structure of this paper is arranged as follows. In section 2, we obtain the mountain pass type solution in the case $\mu\leq 0$ and have completed the proof of Theorem \ref{t1.1}. If $\mu> 0$, there are two situations, that is, Theorems \ref{t1.2} and \ref{t1.3}. Finally, we consider the orbital stability of the solution obtained in Theorems \ref{t1.2} and \ref{t1.4}.

\section{Proof of Theorem \ref{t1.1}}
 In this section, we assume $\mu \leq 0$ and the assumptions of Theorem \ref{t1.1} hold. In order to obtain a bounded Palais-Smale sequence, we will use the monotonicity trick inspired by \cite{LJJ1999}. For $\frac{1}{2} \leq s \leq 1$, we define the functional $I_{r, s}: S_{r, \Theta} \rightarrow \mathbb{R}$ by
\begin{equation}\label{eq3.1}
  I_{r, s}(u)=\frac{1}{2} \int_{\Omega_r\times\mathbb{T}^n}[|\Delta u|^2+V(x,y)u^2] d xdy -\frac{s}{q} \int_{\Omega_r\times\mathbb{T}^n}|u|^q d xdy-\frac{\mu}{p} \int_{\Omega_r\times\mathbb{T}^n}|u|^p d xdy.
\end{equation}
Note that if $u \in S_{r, \Theta}$ is a critical point of $I_{r, s}$, then there exists $\lambda \in \mathbb{R}$ such that $(\lambda, u)$ is a solution of the equation
\begin{equation}\label{eq3.2}
 \left\{\aligned
&\Delta^2 u +V(x,y)u +\lambda u=\mu|u|^{p-2}u+s|u|^{q-2}u,\ (x, y) \in \Omega_r \times \mathbb{T}^n, \\
&\int_{\Omega_r\times\mathbb{T}^n}u^2dxdy=\Theta,\ u\in H_0^2(\Omega_r\times\mathbb{T}^n),
\endaligned
\right.
\end{equation}

\begin{lemma}\label{L3.1}
For any $\Theta>0$, there exist $r_\Theta>0$ and $u^0, u^1 \in S_{r_\Theta, \Theta}$ such that

(i) $I_{r, s}(u^1) \leq 0$ for any $r>r_\Theta$ and $s \in\left[\frac{1}{2}, 1\right]$,
$$
\left\|\Delta u^0\right\|_2^2<\left[\frac{2 q}{(d+n)(q-2) C_{d,n, q}}\left(1-\left\|V_{-}\right\|_{\frac{d+n}{4}} S^{-1}\right) \Theta^{\frac{q(d+n-2)-2 (d+n)}{4}}\right]^{\frac{4}{(d+n)(q-2)-4}}<\left\|\Delta u^1\right\|_2^2
$$
and
$$
I_{r, s}\left(u^0\right)<\frac{((d+n)(q-2)-4)\left(1-\left\|V_{-}\right\|_{\frac{d+n}{4}} S^{-1}\right)}{2 (d+n)(q-2)}\left[\frac{2 q\left(1-\left\|V_{-}\right\|_{\frac{d+n}{4}} S^{-1}\right)}{(d+n)(q-2) C_{d,n, q} \Theta^{\frac{2 q-(d+n)(q-2)}{4}}}\right]^{\frac{4}{(d+n)(q-2)-4}} .
$$

(ii) If $u \in S_{r, \Theta}$ satisfies
$$
\|\Delta u\|_2^2=\left[\frac{2 q}{(d+n)(q-2) C_{d,n, q}}\left(1-\left\|V_{-}\right\|_{\frac{d+n}{4}} S^{-1}\right) \Theta^{\frac{q(d+n-2)-2(d+n)}{4}}\right]^{\frac{4}{(d+n)(q-2)-4}},
$$
then there holds
$$
I_{r, s}(u) \geq \frac{((d+n)(q-2)-4)\left(1-\left\|V_{-}\right\|_{\frac{d+n}{4}} S^{-1}\right)}{2 (d+n)(q-2)}\left[\frac{2 q\left(1-\left\|V_{-}\right\|_{\frac{d+n}{4}} S^{-1}\right)}{(d+n)(q-2) C_{d,n, q} \Theta^{\frac{2 q-(d+n)(q-2)}{4}}}\right]^{\frac{4}{(d+n)(q-2)-4}} .
$$
(iii) Set
$$
m_{r, s}(\Theta)=\inf _{\gamma \in \Gamma_{r, \Theta}} \sup _{t \in[0,1]} I_{r, s}(\gamma(t))
$$
with
$$
\Gamma_{r, \Theta}=\left\{\gamma \in C\left([0,1], S_{r, \Theta}\right): \gamma(0)=u^0, \gamma(1)=u^1\right\} .
$$
Then
\begin{eqnarray*}
&&\frac{((d+n)(q-2)-4)\left(1-\left\|V_{-}\right\|_{\frac{d+n}{4}} S^{-1}\right)}{2(d+n)(q-2)}\left[\frac{2 q\left(1-\left\|V_{-}\right\|_{\frac{d+n}{4}} S^{-1}\right)}{(d+n)(q-2) C_{d,n, q} \Theta^{\frac{2 q-(d+n)(q-2)}{4}}}\right]^{\frac{4}{(d+n)(q-2)-4}} \\
&\leq& m_{r, s}(\Theta) \leq h(T_\Theta),
\end{eqnarray*}
where $h(T_\Theta)=\max\limits_{t \in \mathbb{R}^{+}} h(t)$, the function $h: \mathbb{R}^{+} \rightarrow \mathbb{R}$ being defined by
\begin{eqnarray*}
h(t)=\frac{1}{2}\left(1+\|V\|_{\frac{d+n}{4}} S^{-1}\right) t^4 \theta \Theta-\alpha\beta C_{d,n, p} \Theta^{\frac{p}{2}} \theta^{\frac{(d+n)(p-2)}{4}} t^{\frac{(d+n)(p-2)}{2}}-\frac{1}{2 q} \Theta^{\frac{q}{2}}|\Omega|^{\frac{2-q}{2}} t^{\frac{(d+n)(q-2)}{2}} .
\end{eqnarray*}
Here $\theta$ is the principal eigenvalue of $\Delta^2$ with Dirichlet boundary conditions in $\Omega\times\mathbb{T}^n$, and $|\Omega\times\mathbb{T}^n|$ is the volume of $\Omega\times\mathbb{T}^n$.
\end{lemma}
\begin{proof}
(i) Clearly, the set $S_{r, \Theta}$ is path connected. Since $v_1 \in S_{1, \Theta}$ be the positive eigenfunction associated to $\theta$ and note that $\theta$ is the principal eigenvalue of $\Delta^2$, then
\begin{equation}\label{eq3.3}
  \int_{\Omega\times\mathbb{T}^n}\left|\Delta v_1\right|^2 d xdy=\theta \Theta.
\end{equation}
By the H\"older inequality, we know that
\begin{eqnarray*}
  \Theta=\int_{\Omega\times\mathbb{T}^n}|v_1(x,y)|^2dxdy  \leq\left(\int_{\Omega\times\mathbb{T}^n}|v_1(x,y)|^{q} dxdy\right)^{\frac{2}{q}}\cdot|\Omega\times\mathbb{T}^n|^{\frac{q-2}{q}},
\end{eqnarray*}
which implies
\begin{equation}\label{eq3.4}
  \int_{\Omega\times\mathbb{T}^n}|v_1(x,y)|^{q} dxdy\geq\Theta^{\frac{q}{2}}\cdot|\Omega\times\mathbb{T}^n|^{\frac{2-q}{2}}.
\end{equation}
For $(x,y) \in \Omega_{\frac{1}{t}}\times\mathbb{T}_{\frac{1}{t}}^n$ and $t>0$, define $v_t(x,y):=t^{\frac{d+n}{2}} v_1(t x, ty)$. Using \eqref{eq3.3}, \eqref{eq3.4} and $\frac{1}{2}\leq s\leq1$, it holds
\begin{eqnarray}\label{eq3.6}
&&I_{\frac{1}{t}, s}\left(v_t\right)\nonumber\\
&\leq&\frac{1}{2} \int_{\Omega_{\frac{1}{t}}\times\mathbb{T}_{\frac{1}{t}}^n}[|\Delta v_t|^2+V(x,y)v_t^2] d xdy-\frac{1}{2q} \int_{\Omega_{\frac{1}{t}}\times\mathbb{T}_{\frac{1}{t}}^n}|v_t|^q d xdy- \frac{\mu}{p}  \int_{\Omega_{\frac{1}{t}}\times\mathbb{T}_{\frac{1}{t}}^n}|v_t|^{p} d xdy\nonumber\\
&\leq&\frac{1}{2}\left(1+\|V\|_{\frac{d+n}{4}} S^{-1}\right) \int_{\Omega_{\frac{1}{t}}\times\mathbb{T}_{\frac{1}{t}}^n}|\Delta v_t|^2 d x-\frac{1}{2q} \int_{\Omega_{\frac{1}{t}}\times\mathbb{T}_{\frac{1}{t}}^n}|v_t|^q d xdy\nonumber\\
&&-\frac{\mu}{p} C_{d,n, p} \Theta^{\frac{4 p-(d+n)(p-2)}{8}}\left( \int_{\Omega_{\frac{1}{t}}\times\mathbb{T}_{\frac{1}{t}}^n}\left|\Delta v_t\right|^2 d xdy\right)^{\frac{(d+n)(p-2)}{8}}\nonumber\\
&\leq&\frac{1}{2}\left(1+\|V\|_{\frac{d+n}{4}} S^{-1}\right) t^4 \int_{\Omega\times\mathbb{T}^n}\left|\Delta v_1\right|^2 d xdy-\frac{1}{2q} t^{\frac{(d+n)(q-2)}{2}} \int_{\Omega\times\mathbb{T}^n}\left|v_1\right|^q d xdy \nonumber\\
&&-\frac{\mu}{p} C_{d,n, p} \Theta^{\frac{4 p-(d+n)(p-2)}{8}}\left(t^4 \int_{\Omega\times\mathbb{T}^n}\left|\Delta v_1\right|^2 d xdy\right)^{\frac{(d+n)(p-2)}{8}} \nonumber\\
&\leq & \frac{1}{2}\left(1+\|V\|_{\frac{d+n}{4}} S^{-1}\right) t^4 \theta\Theta-\frac{\mu}{p} C_{d,n, p} \Theta^{\frac{p}{2}}\theta^{\frac{(d+n)(p-2)}{8}}t^{\frac{(d+n)(p-2)}{2}}-\frac{1}{2q} t^{\frac{(d+n)(q-2)}{2}} \Theta^{\frac{q}{2}}\cdot|\Omega|^{\frac{2-q}{2}} \nonumber\\
&=: & h(t) .
\end{eqnarray}
Note that since $2<p<2+\frac{8}{d+n}<q<4^*$ and $\mu\leq0$ there exist $0<T_\Theta<t_0$ such that $h\left(t_0\right)=0, h(t)<0$ for any $t>t_0, h(t)>0$ for any $0<t<t_0$ and $h\left(T_\Theta\right)=\max\limits_{t \in \mathbb{R}^{+}} h(t)$. As a consequence, there holds
\begin{equation}\label{eq3.7}
  I_{r, s}\left(v_{t_0}\right)=I_{\frac{1}{t_0}, s}\left(v_{t_0}\right) \leq h\left(t_0\right)=0
\end{equation}
for any $r \geq \frac{1}{t_0}$ and $s \in\left[\frac{1}{2}, 1\right]$. Moreover, there exists $0<t_1<T_\Theta$ such that
\begin{equation}\label{eq3.8}
 h(t)<\frac{((d+n)(q-2)-4)\left(1-\left\|V_{-}\right\|_{\frac{d+n}{4}} S^{-1}\right)}{2 (d+n)(q-2)}\left[\frac{2 q\left(1-\left\|V_{-}\right\|_{\frac{d+n}{4}} S^{-1}\right)}{(d+n)(q-2) C_{d,n, q} \Theta^{\frac{2 q-(d+n)(q-2)}{4}}}\right]^{\frac{4}{(d+n)(q-2)-4}}
\end{equation}
for $t \in\left[0, t_1\right]$. On the other hand, it follows from the Gagliardo-Nirenberg inequality and the H\"older inequality that
\begin{eqnarray}\label{eq3.9}
&&I_{r, s}(u)\nonumber\\
&\geq&\frac{1}{2} \int_{\Omega_r\times\mathbb{T}^n}|\Delta u|^2 d xdy+\frac{1}{2} \int_{\Omega_r\times\mathbb{T}^n} V u^2 d xdy-\frac{1}{q} \int_{\Omega_r\times\mathbb{T}^n}|u|^q d xdy\nonumber\\
&\geq& \frac{1-\left\|V_{-}\right\|_{\frac{d+n}{4}} S^{-1}}{2} \int_{\Omega_r\times\mathbb{T}^n}|\Delta u|^2 d xdy-\frac{C_{d,n, q} \Theta^{\frac{4 q-(d+n)(q-2)}{8}}}{q}\left(\int_{\Omega_r\times\mathbb{T}^n}|\Delta u|^2 d xdy\right)^{\frac{(d+n)(q-2)}{8}}.
\end{eqnarray}
Define
$$
g(t):=\frac{1}{2}\left(1-\left\|V_{-}\right\|_{\frac{d+n}{4}} S^{-1}\right) t-\frac{C_{d,n, q} \Theta^{\frac{4 q-(d+n)(q-2)}{8}}}{q} t^{\frac{(d+n)(q-2)}{8}}
$$
and
$$
\widetilde{t}=\left[\frac{4 q}{(d+n)(q-2) C_{d,n, q}}\left(1-\left\|V_{-}\right\|_{\frac{d+n}{4}} S^{-1}\right) \Theta^{\frac{q(d+n-4)-2(d+n)}{8}}\right]^{\frac{8}{(d+n)(q-2)-8}},
$$
it is easy to see that $g$ is increasing on $(0, \widetilde{t})$ and decreasing on $(\widetilde{t}, \infty)$, and
$$
g(\widetilde{t})=\frac{((d+n)(q-2)-4)\left(1-\left\|V_{-}\right\|_{\frac{d+n}{4}} S^{-1}\right)}{2 (d+n)(q-2)}\left[\frac{4 q\left(1-\left\|V_{-}\right\|_{\frac{d+n}{4}} S^{-1}\right)}{(d+n)(q-2) C_{d,n, q} \Theta^{\frac{4 q-(d+n)(q-2)}{8}}}\right]^{\frac{8}{(d+n)(q-2)-8}} .
$$
For $r \geq \widetilde{r}_\Theta:=\max \left\{\frac{1}{t_1}, \sqrt{\frac{2 \theta \Theta}{\tilde{t}}}\right\}$, we have $v_{\frac{1}{\tilde{r}_\Theta}} \in S_{r, \Theta}$ and
\begin{eqnarray}\label{eq3.10}
\|\Delta v_{\frac{1}{\widetilde{r}_\Theta}}\|_2^2&=&\left(\frac{1}{\widetilde{r}_\Theta}\right)^4\left\|\Delta v_1\right\|_2^2\nonumber\\
&<&\left[\frac{4 q}{(d+n)(q-2) C_{d,n, q}}\left(1-\left\|V_{-}\right\|_{\frac{d+n}{4}} S^{-1}\right) \Theta^{\frac{q(d+n-4)-2(d+n)}{8}}\right]^{\frac{8}{(d+n)(q-2)-8}} .
\end{eqnarray}
Moreover, there holds
\begin{equation}\label{eq3.11}
 I_{\widetilde{r}_\Theta, s}\left(v_{\frac{1}{\widetilde{r}_\Theta}}\right) \leq h\left(\frac{1}{\widetilde{r}_\Theta}\right) \leq h\left(t_1\right) .
\end{equation}
Setting $u^0=v_{\frac{1}{\tilde{r}_\Theta}}, u^1=v_{t_0}$ and
\begin{equation}\label{eq3.12}
  r_\Theta=\max \left\{\frac{1}{t_0}, \widetilde{r}_\Theta\right\}.
\end{equation}
Combining \eqref{eq3.7}, \eqref{eq3.8}, \eqref{eq3.10} and \eqref{eq3.11}, (i) holds.

(ii) By \eqref{eq3.9} and a direct calculation, (ii) holds.

(iii) Since $I_{r, s}\left(u^1\right) \leq 0$ for any $\gamma \in \Gamma_{r, \Theta}$, we have
$$
\|\Delta \gamma(0)\|_2^2<\tilde{t}<\|\Delta \gamma(1)\|_2^2.
$$
It then follows from \eqref{eq3.9} that
$$
\begin{aligned}
&\max\limits_{t \in[0,1]} I_{r, s}(\gamma(t)) \\
& \geq g(\widetilde{t}) \\
& =\frac{((d+n)(q-2)-4)\left(1-\left\|V_{-}\right\|_{\frac{d+n}{4}} S^{-1}\right)}{2 (d+n)(q-2)}\left[\frac{4 q\left(1-\left\|V_{-}\right\|_{\frac{d+n}{4}} S^{-1}\right)}{(d+n)(q-2) C_{d,n, q} \Theta^{\frac{4 q-(d+n)(q-2)}{8}}}\right]^{\frac{8}{(d+n)(q-2)-8}}
\end{aligned}
$$
for any $\gamma \in \Gamma_{r, \Theta}$, hence the first inequality in (iii) holds. Now we define a path $\gamma \in \Gamma_{r, \Theta}$ by
$$
\gamma(\tau)(x,y)=\left(\tau t_0+(1-\tau) \frac{1}{\widetilde{r}_\Theta}\right)^{\frac{d+n}{2}} v_1\left(\left(\tau t_0+(1-\tau) \frac{1}{\widetilde{r}_\Theta}\right) x,\left(\tau t_0+(1-\tau) \frac{1}{\widetilde{r}_\Theta}\right) y\right)
$$
for $\tau \in[0,1]$ and $x \in \Omega_r$. Then by \eqref{eq3.6} we have $m_{r, s}(\Theta) \leq h(T_\Theta)$, where $h(T_\Theta)=\max\limits_{t \in \mathbb{R}^{+}} h(t)$. Note that $T_\Theta$ is independent of $r$ and $s$.
\end{proof}

By using Lemma \ref{L3.1}, the energy functional $I_{r, s}$ possesses the mountain pass geometry. To obtain bounded Palais-Smale sequence, we recall a proposition from \cite{{JBXC2022},{XCLJ2022}}.

\begin{proposition}\label{p3.1}
{\em(see \cite[Theorem 1]{{JBXC2022}})} Let $(E,\langle\cdot, \cdot\rangle)$ and $(H,(\cdot, \cdot))$ be two infinite-dimensional Hilbert spaces and assume there are continuous injections
$$
E \hookrightarrow H \hookrightarrow E^{\prime} .
$$
Let
$$
\|u\|^2=\langle u, u\rangle, \quad|u|^2=(u, u) \quad \text { for } u \in E,
$$
and
$$
S_\mu=\left\{u \in E:|u|^2=\mu\right\}, \quad T_u S_\mu=\{v \in E:(u, v)=0\} \quad \text { for } \mu \in(0,+\infty) .
$$
Let $I \subset(0,+\infty)$ be an interval and consider a family of $C^2$ functionals $\Phi_\rho: E \rightarrow \mathbb{R}$ of the form
$$
\Phi_\rho(u)=A(u)-\rho B(u), \quad \text { for } \rho \in I,
$$
with $B(u) \geq 0$ for every $u \in E$, and
\begin{equation}\label{eq3.13}
  A(u) \rightarrow+\infty \quad \text { or } \quad B(u) \rightarrow+\infty \quad \text { as } u \in E \text { and }\|u\| \rightarrow+\infty.
\end{equation}
Suppose moreover that $\Phi_\rho^{\prime}$ and $\Phi_\rho^{\prime \prime}$ are $\tau$-H\"older continuous, $\tau \in(0,1]$, on bounded sets in the following sense: for every $R>0$ there exists $M=M(R)>0$ such that
\begin{equation}\label{eq3.14}
\left\|\Phi_\rho^{\prime}(u)-\Phi_\rho^{\prime}(v)\right\| \leq M\|u-v\|^\tau \quad \text { and } \quad\left\|\Phi_\rho^{\prime \prime}(u)-\Phi_\rho^{\prime \prime}(v)\right\| \leq M\|u-v\|^\tau
\end{equation}
for every $u, v \in B(0, R)$. Finally, suppose that there exist $w_1, w_2 \in S_\mu$ independent of $\rho$ such that
$$
c_\rho:=\inf _{\gamma \in \Gamma} \max _{t \in[0,1]} \Phi_\rho(\gamma(t))>\max \left\{\Phi_\rho\left(w_1\right), \Phi_\rho\left(w_2\right)\right\} \quad \text { for all } \rho \in I,
$$
where
$$
\Gamma=\left\{\gamma \in C\left([0,1], S_\mu\right): \gamma(0)=w_1, \gamma(1)=w_2\right\} .
$$
Then for almost every $\rho \in I$, there exists a sequence $\left\{u_n\right\} \subset S_\mu$ such that

(i) $\Phi_\rho\left(u_n\right) \rightarrow c_\rho$,

(ii) $\left.\Phi_\rho^{\prime}\right|_{S_\mu}\left(u_n\right) \rightarrow 0$,

(iii) $\left\{u_n\right\}$ is bounded in $E$.
\end{proposition}

\begin{lemma}\label{L3.2}
For any $\Theta>0$, let $r>r_\Theta$, where $r_\Theta$ is defined in Lemma \ref{L3.1}. Then problem \eqref{eq3.1} has a solution $\left(\lambda_{r, s}, u_{r, s}\right)$ for almost every $s \in\left[\frac{1}{2}, 1\right]$. Moreover, $u_{r, s} \geq 0$ and $I_{r, s}\left(u_{r, s}\right)=m_{r, s}(\Theta)$.
\end{lemma}
\begin{proof}
By Proposition \ref{p3.1}, it follows that
$$
A(u)=\frac{1}{2} \int_{\Omega_r\times\mathbb{T}^n}|\Delta u|^2 d xdy+\frac{1}{2} \int_{\Omega_r\times\mathbb{T}^n} V(x) u^2 d xdy-\frac{\mu}{p} \int_{\Omega_r\times\mathbb{T}^n}|u|^p d xdy
$$
and
\begin{equation*}
  B(u)=\frac{1}{q} \int_{\Omega_r\times\mathbb{T}^n}|u|^q d xdy.
\end{equation*}
Note that the assumptions in Proposition \ref{p3.1} hold due to $\mu \leq 0$ and Lemma \ref{L3.1}. Hence, for almost every $s \in\left[\frac{1}{2}, 1\right]$, there exists a bounded Palais-Smale sequence $\left\{u_n\right\}$ satisfying
$$
I_{r, s}\left(u_n\right) \rightarrow m_{r, s}(\Theta) \quad \text { and }\left.\quad I_{r, s}^{\prime}\left(u_n\right)\right|_{T_{u_n} S_{r, \Theta}} \rightarrow 0,
$$
where $T_{u_n} S_{r, \Theta}$ denotes the tangent space of $S_{r, \Theta}$ at $u_n$. Then
$$
\lambda_n=-\frac{1}{\Theta}\left(\int_{\Omega_r\times\mathbb{T}^n}\left(\left|\Delta u_n\right|^2 +V(x) u_n^2\right) d xdy-\mu \int_{\Omega_r\times\mathbb{T}^n}|u_n|^p d xdy-s \int_{\Omega_r\times\mathbb{T}^n}\left|u_n\right|^q d xdy\right)
$$
is bounded and
\begin{equation}\label{eq3.15}
  I_{r, s}^{\prime}\left(u_n\right)+\lambda_n u_n \rightarrow 0 \quad \text { in } H^{-2}\left(\Omega_r\times\mathbb{T}^n\right) .
\end{equation}
Moreover, since $\left\{u_n\right\}$ is a bounded Palais-Smale sequence, there exist $u_0 \in H_0^1\left(\Omega_r\times\mathbb{T}^n\right)$ and $\lambda \in \mathbb{R}$ such that, up to a subsequence,
\begin{eqnarray*}
\lambda_n &\rightarrow& \lambda \ \text { in }\ \mathbb{R},\\
u_n &\rightharpoonup& u_0 \  \text { in }\ H_0^2(\Omega_r\times\mathbb{T}^n),\\
 u_n &\rightarrow& u_0 \  \text { in }\ L^t(\Omega_r\times\mathbb{T}^n) \text { for all } 2 \leq t<4^*,
\end{eqnarray*}
where $u_0$ satisfies
$$
\left\{\begin{array}{l}
\Delta^2 u_0+V u_0+\lambda u_0=s\left|u_0\right|^{q-2} u_0+\mu \left|u_0\right|^{p-2} u_0 \quad \text { in } \Omega_r\times\mathbb{T}^n, \\
u_0 \in H_0^2\left(\Omega_r\times\mathbb{T}^n\right), \quad \int_{\Omega_r\times\mathbb{T}^n}\left|u_0\right|^2 d xdy=\Theta.
\end{array}\right.
$$
Using \eqref{eq3.15}, we have
$$
I_{r, s}^{\prime}\left(u_n\right) u_0+\lambda_n \int_{\Omega_r\times\mathbb{T}^n} u_n u_0 d xdy \rightarrow 0 \text { as } n \rightarrow \infty
$$
and
\begin{equation*}
  I_{r, s}^{\prime}\left(u_n\right) u_n+\lambda_n \Theta\rightarrow 0 \ \text{as }\  n \rightarrow \infty  .
\end{equation*}
Note that
\begin{eqnarray*}
\lim _{n \rightarrow \infty} \int_{\Omega_r\times\mathbb{T}^n} V(x,y) u_n^2 d x&=&\int_{\Omega_r\times\mathbb{T}^n} V(x,y) u_0^2 d xdy,\\
\lim _{n \rightarrow \infty} \int_{\Omega_r\times\mathbb{T}^n} |u_n|^p d x&=&\int_{\Omega_r\times\mathbb{T}^n} |u_0|^p d x,
\end{eqnarray*}
so we get $u_n \rightarrow u_0$ in $H_0^2(\Omega_r\times\mathbb{T}^n)$, hence $I_{r, s}(u_0)=m_{r, s}(\Theta)$.
\end{proof}

\begin{remark}\label{R3.1}
If we consider schr\"odinger operator $-\Delta$, then we can get better result, that is $u_{r, s} \geq 0$, see \cite{TBAQ2023}.
\end{remark}
\begin{remark}\label{R3.2}
In this paper, we consider the biharmonic operator $\Delta^2$  which is more complex than $-\Delta$, so we can not get that \eqref{eq3.1} has a nonnegative normalized solution. Precisely speaking, we use the fact that $\left\|\nabla|u|\|_2^2 \leq\right. \|\nabla u\|_2^2$ for any $u \in H^1(\mathbb{R}^d\times\mathbb{T}^n)$ in Remark \ref{R3.1}.  However, we do not know the size relationship between $\|\Delta u\|_2^2$ and $\|\Delta|u|\|_2^2$.
\end{remark}

In order to obtain a solution of \eqref{eq1.1}, we need to prove a uniform estimate for the solutions of \eqref{eq3.1} established in Lemma \ref{L3.2}.

\begin{lemma}\label{L3.3}
If $(\lambda_{r, s}, u_{r, s}) \in \mathbb{R} \times S_{r, \Theta}$ is a solution of \eqref{eq3.1} established in Lemma \ref{L3.2} for some $r$ and $s$, then
$$
\int_{\Omega_r\times\mathbb{T}^n}|\Delta u|^2 d xdy \leq \frac{4(d+n)}{(d+n)(q-2)-4}\left[\frac{q-2}{2} h(T_\Theta)+\Theta\left(\frac{1}{2 (d+n)}\|\widetilde{V}\|_{\infty}+\frac{q-2}{4}\|V\|_{\infty}\right)\right],
$$
where the constant $h(T_\Theta)$ is defined in (iii) of Lemma \ref{L3.1}  and is independent of $r$ and $s$.
\end{lemma}
\begin{proof} For simplicity, we denote $(\lambda_{r, s}, u_{r, s})$ as $(\lambda, u)$ in this lemma. Since $u$ is a solution of \eqref{eq3.1}, we have
\begin{equation}\label{eq3.16}
  \int_{\Omega_r\times\mathbb{T}^n}(|\Delta u|^2+ V u^2) d xdy=s \int_{\Omega_r\times\mathbb{T}^n}|u|^q d xdy+\mu \int_{\Omega_r\times\mathbb{T}^n}|u|^p d xdy-\lambda \int_{\Omega_r\times\mathbb{T}^n}|u|^2 d xdy .
\end{equation}
The Pohozaev identity implies
\begin{eqnarray*}
&&\frac{d+n-4}{2(d+n)} \int_{\Omega_r\times\mathbb{T}^n}|\Delta u|^2 d xdy+\frac{1}{2(d+n)} \int_{\partial (\Omega_r\times\mathbb{T}^n)}|\Delta u|^2((x,y) \cdot \mathbf{n}) d \sigma \\
&&+\frac{1}{2(d+n)} \int_{\Omega_r\times\mathbb{T}^n}\widetilde{V}(x,y) u^2dxdy+\frac{1}{2} \int_{\Omega_r\times\mathbb{T}^n} V u^2 d xdy\\
&=&-\frac{\lambda}{2} \int_{\Omega_r\times\mathbb{T}^n}|u|^2 d xdy+\frac{s}{q} \int_{\Omega_r\times\mathbb{T}^n}|u|^q d xdy+\frac{\mu}{p} \int_{\Omega_r\times\mathbb{T}^n}|u|^p d xdy,
\end{eqnarray*}
where $\mathbf{n}$ denotes the outward unit normal vector on $\partial(\Omega_r\times\mathbb{T}^n)$. It then follows from $\mu \leq 0$ that
\begin{eqnarray*}
&&\frac{2}{d+n} \int_{\Omega_r\times\mathbb{T}^n}|\Delta u|^2 d  xdy-\frac{1}{2(d+n)} \int_{\partial (\Omega_r\times\mathbb{T}^n)}|\nabla u|^2((x,y) \cdot \mathbf{n}) d \sigma\\
&&-\frac{1}{2(d+n)} \int_{\Omega_r\times\mathbb{T}^n}(\nabla V \cdot (x,y)) u^2 d xdy \\
& =&\frac{(q-2) s}{2 q} \int_{\Omega_r\times\mathbb{T}^n}|u|^q d xdy+ \mu\int_{\Omega_r\times\mathbb{T}^n}(\frac{1}{2}-\frac{1}{p})|u|^pd xdy \\
& \geq& \frac{(q-2) s}{2 q} \int_{\Omega_r\times\mathbb{T}^n}|u|^q d xdy+  \frac{\mu (q-2) }{2p  }  \int_{\Omega_r}|u|^p d xdy \\
& =&\frac{q-2}{2}\left(\frac{1}{2} \int_{\Omega_r\times\mathbb{T}^n}|\Delta u|^2 d xdy+\frac{1}{2} \int_{\Omega_r\times\mathbb{T}^n} V u^2 d xdy-m_{r, s}(\Theta)\right) .
\end{eqnarray*}
Consequently, we have
 \begin{eqnarray*}
&&\frac{q-2}{2} m_{r, s}(\Theta) \\
&\geq & \frac{q-2}{2}\left(\frac{1}{2} \int_{\Omega_r\times\mathbb{T}^n}|\Delta u|^2 d xdy+\frac{1}{2} \int_{\Omega_r} V u^2 d xdy\right)-\frac{2}{d+n} \int_{\Omega_r\times\mathbb{T}^n}|\Delta u|^2 d xdy \\
&& +\frac{1}{2(d+n)} \int_{\partial (\Omega_r\times\mathbb{T}^n)}|\nabla u|^2((x,y) \cdot \mathbf{n}) d \sigma+\frac{1}{2(d+n)} \int_{\Omega_r\times\mathbb{T}^n}(\nabla V \cdot (x,y)) u^2 d xdy \\
&\geq & \frac{(d+n)(q-2)-8}{4(d+n)} \int_{\Omega_r\times\mathbb{T}^n}|\Delta u|^2 d xdy-\Theta\left(\frac{1}{2(d+n)}\|\nabla V \cdot (x,y)\|_{\infty}+\frac{q-2}{4}\|V\|_{\infty}\right),
 \end{eqnarray*}
where the last inequality holds since $(x,y) \cdot \mathbf{n}  \geq 0$ for any $(x,y) \in \partial (\Omega_r\times\mathbb{T}^n)$ due to the convexity of $\Omega_r$. Using Lemma \ref{L3.1}, we have
\begin{eqnarray*}
 &&\frac{(d+n)(q-2)-8}{4(d+n)} \int_{\Omega_r\times\mathbb{T}^n}|\Delta u|^2 d xdy-\Theta\left(\frac{1}{2(d+n)}\|\nabla V \cdot x\|_{\infty}+\frac{q-2}{4}\|V\|_{\infty}\right)\\
 &\leq&\frac{q-2}{2}h(T_\Theta),
\end{eqnarray*}
which implies
\begin{equation*}
  \int_{\Omega_r\times\mathbb{T}^n}|\Delta u|^2 d xdy \leq \frac{4(d+n)}{(d+n)(q-2)-4}\left[\frac{q-2}{2} h(T_\Theta)+\Theta\left(\frac{1}{2 (d+n)}\|\widetilde{V}\|_{\infty}+\frac{q-2}{4}\|V\|_{\infty}\right)\right].
\end{equation*}
This completes the proof of lemma.
\end{proof}

Now, we obtain a solution of \eqref{eq1.1} by letting $s \rightarrow 1$.

\begin{lemma}\label{L3.4}
For every $\Theta>0$, problem \eqref{eq1.1} has a solution $\left(\lambda_r, u_r\right)$ provided $r>r_\Theta$ where $r_\Theta$ is as in Lemma \ref{L3.1}.
\end{lemma}
\begin{proof}
By using Lemma  \ref{L3.2}, there is a nontrivial solution $(\lambda_{r, s}, u_{r, s})$ to \eqref{eq3.1} for almost every $s \in\left[\frac{1}{2}, 1\right]$. In view of Lemma  \ref{L3.3}, $\left\{u_{r, s}\right\}$ is bounded. By an argument similar to that in Lemma  \ref{L3.2}, there exist $u_r \in S_{r, \Theta}$ and $\lambda_r$ such that, going if necessary to a subsequence,
$$
\lambda_{r, s} \rightarrow \lambda_r \quad \text { and } \quad u_{r, s} \rightarrow u_r \quad \text { in } H_0^2\left(\Omega_r\right) \quad \text { as } s \rightarrow 1 .
$$
Hence $u_r$ is a nontrivial solution of problem \eqref{eq1.1}.
\end{proof}

Next, we will consider the Lagrange multiplier. we first establish an a priori estimate for the solutions of \eqref{eq1.1}.

\begin{lemma}\label{L3.5}
If $\left\{\left(\lambda_r, u_r\right)\right\}$ is a family of nontrivial solutions of \eqref{eq1.1} such that $\left\|u_r\right\|_{H^2} \leq$ $C$ with $C>0$ independent of $r$, then $\limsup\limits_{r \rightarrow \infty}\left\|u_r\right\|_{\infty}<\infty$.
\end{lemma}
\begin{proof}
Using the regularity theory of elliptic partial differential equations, we know that $u_r \in C(\Omega_r\times\mathbb{T}^n)$. Assume to the contrary that there exist a sequence, for simplicity denoted by $\left\{u_r\right\}$, and $(x_r,y_r) \in \Omega_r\times\mathbb{T}^n$ such that
$$
M_r:=\max _{(x,y) \in \Omega_r\times\mathbb{T}^n} u_r(x,y)=u_r\left(x_r,y_r\right) \rightarrow \infty \quad \text { as } r \rightarrow \infty.
$$
Suppose without loss of generality that, up to a subsequence, $\lim\limits_{r \rightarrow \infty} \frac{x_r}{\left|x_r\right|}=\frac{y_r}{\left|y_r\right|}=(1,0, \ldots, 0)$. Set
$$
v_r(x)=\frac{u_r\left(x_r+\tau_r x,y_r+\tau_r y\right)}{M_r}
$$
for $(x,y) \in \Sigma^r:=\left\{(x,y) \in \mathbb{R}^d\times\mathbb{T}^n: (x_r+\tau_r x,y_r+\tau_r y)\in \Omega_r\times\mathbb{T}^n\right\}$, where $\tau_r=M_r^{\frac{2-q}{4}}$. Then $\tau_r \rightarrow 0$ as $r \rightarrow \infty$, $\left\|v_r\right\|_{L^{\infty}\left(\Sigma^r\right)} \leq 1$, and $v_r$ satisfies
\begin{equation}\label{eq3.17}
\Delta^2 v_r+\tau_r^4 V\left(x_r+\tau_r x,y_r+\tau_r y\right) v_r+\tau_r^4 \lambda_r v_r=\left|v_r\right|^{q-2} v_r+\mu \tau_r^{\frac{4(q-p)}{q-2}}\left|v_r\right|^{p-2} v_r \quad \text { in } \Sigma^r .
\end{equation}
In fact, since $u_r$ is a nontrivial solution of \eqref{eq1.1}, we obtain
\begin{eqnarray*}
&&\Delta^2 u_r\left(x_r+\tau_r x,y_r+\tau_r y\right)+ V\left(x_r+\tau_r x,y_r+\tau_r y\right) u_r\left(x_r+\tau_r x,y_r+\tau_r y\right)\\
&&+ \lambda_r u_r\left(x_r+\tau_r x,y_r+\tau_r y\right)\\
&=&\left|u_r\left(x_r+\tau_r x,y_r+\tau_r y\right)\right|^{q-2} u_r\left(x_r+\tau_r x,y_r+\tau_r y\right)\\
&&+\mu\left|u_r\left(x_r+\tau_r x,y_r+\tau_r y\right)\right|^{p-2} u_r\left(x_r+\tau_r x,y_r+\tau_r y\right)
\end{eqnarray*}
in $\Omega_r$, then by a direct calculation and the definition of $v_r(x)$, $\tau_r$, we know that \eqref{eq3.17} holds.
In view of \eqref{eq1.1}, the Gagliardo-Nirenberg inequality and $\left\|u_r\right\|_{H^2} \leq C$ with $C$ independent of $r$, we infer that the sequence $\left\{\lambda_r\right\}$ is bounded. It then follows from the regularity theory of elliptic partial differential equations and the Arzela-Ascoli theorem that there exists $v$ such that, up to a subsequence
$$
v_r \rightarrow v \quad \text { in } H_0^2(\Sigma) \quad \text { and } \quad v_r \rightarrow v \quad \text { in } C_{l o c}^\beta(\Sigma) \text { for some } \beta \in(0,1),
$$
where $\Sigma:=\lim\limits_{r \rightarrow \infty} \Sigma^r$.

Similar to the proof of \cite[Lemma 2.7]{{TBAQ2023}}, we have
 \begin{equation*}
   \liminf _{r \rightarrow \infty} \frac{\operatorname{dist}\left((x_r,y_r), \partial (\Omega_r\times\mathbb{T}^n)\right)}{\tau_r}=\liminf _{r \rightarrow \infty} \frac{\left|z_r-x_r\right|}{\tau_r} \geq d_1>0,
 \end{equation*}
where $z_r \in \partial (\Omega_r\times\mathbb{T}^n)$ is such that $\operatorname{dist}\left(x_r, \partial(\Omega_r\times\mathbb{T}^n)\right)=\left|z_r-x_r\right|$ for any large $r$. As a result, by letting $r \rightarrow \infty$ in \eqref{eq3.17}, we obtain that $v \in H_0^2(\Sigma)$ is a nontrivial solution of
$$
\Delta^2 v=|v|^{q-2} v \quad \text { in } \Sigma,
$$
where
$$
\Sigma= \begin{cases}\mathbb{R}^d\times\mathbb{T}^n & \text { if } \liminf\limits_{r \rightarrow \infty} \frac{\operatorname{dist}\left((x_r,y_r), \partial (\Omega_r\times\mathbb{T}^n)\right)}{\tau_r}=\infty, \\ \left\{x \in \mathbb{R}^d: x_1>-d_1\right\}\times\mathbb{T}^n & \text { if } \liminf\limits_{r \rightarrow \infty} \frac{\operatorname{dist}\left((x_r,y_r), \partial(\Omega_r\times\mathbb{T}^n)\right)}{\tau_r}>0 .\end{cases}
$$
It then follows from the Liouville theorems (see \cite{MEPL1982}) that $v=0$ in $H_0^2(\Sigma)$, which contradicts $v(0)=\lim\limits_{r \rightarrow \infty} v_r(0)=1$.
\end{proof}

Clearly, the proof of Lemma \ref{L3.5} does not depend on $\mu$.

\begin{lemma}\label{L3.6}
Let $\left(\lambda_{r, \Theta}, u_{r, \Theta}\right)$ be the solution of \eqref{eq1.1} from Lemma \ref{L3.4}. If $\|\widetilde{V}_{+}\|_{\frac{d+n}{4}} < 2 S$, then there exists $\bar{\Theta}>0$ such that
$$
\liminf\limits_{r \rightarrow \infty} \lambda_{r, \Theta}>0 \quad \text { for } 0<\Theta<\bar{\Theta} .
$$
\end{lemma}
\begin{proof}
Let $\left(\lambda_{r, \Theta}, u_{r, \Theta}\right)$ be the solution of \eqref{eq1.1} established in Theorem \ref{L3.4}. By the regularity theory of elliptic partial differential equations, we have $u_{r, \Theta} \in C\left(\Omega_r\times\mathbb{T}^n\right)$. Using Lemma \ref{L3.5}, it holds
$$
\limsup _{r \rightarrow \infty} \max\limits_{\Omega_r\times\mathbb{T}^n} u_{r, \Theta}<\infty .
$$
Setting
$$
Q(\Theta)=\liminf\limits_{r \rightarrow \infty} \max\limits_{\Omega_r\times\mathbb{T}^n} u_{r, \Theta},
$$
we claim that there is $\Theta_1>0$ such that $Q(\Theta)>0$ for any $0<\Theta<\Theta_1$. Assume to the contrary that there exists a sequence $\left\{\Theta_k\right\}$ tending to 0 as $k \rightarrow \infty$ such that $Q\left(\Theta_k\right)=0$ for any $k$, that is,
\begin{equation}\label{eq3.20}
  \liminf\limits_{r \rightarrow \infty} \max\limits_{\Omega_r\times\mathbb{T}^n} u_{r, \Theta_k}=0 \quad \text {for any}\ k.
\end{equation}
As a consequence of (iii) in Lemma \ref{L3.1}, for any $r>r_{\Theta_k}$, we have
\begin{equation}\label{eq3.21}
I_r\left(u_{r, \Theta_k}\right)=m_{r, 1}\left(\Theta_k\right) \rightarrow \infty \quad \text { as } k \rightarrow \infty .
\end{equation}
For any given $k$, it follows from \eqref{eq3.20} and $u_{r, \Theta_k} \in S_{r, \Theta_k}$ that, up to a subsequence,
\begin{equation}\label{eq3.22}
\int_{\Omega_r\times\mathbb{T}^n}\left|u_{r, \Theta_k}\right|^s d xdy=\int_{\Omega_r\times\mathbb{T}^n}\left|u_{r, \Theta_k}\right|^{s-2}\left|u_{r, \Theta_k}\right|^2 d xdy \leq\left|\max _{\Omega_r\times\mathbb{T}^n} u_{r, \Theta_k}\right|^{s-2} \Theta_k \rightarrow 0
\end{equation}
as $r\rightarrow \infty$ for any $s>2$. Hence, for any given large $k$, there exists $\bar{r}_k>r_{\Theta_k}$ such that
$$
\left|\frac{1}{q} \int_{\Omega_r\times\mathbb{T}^n} |u_{r, \Theta_k}|^q d xdy+\frac{\mu}{p} \int_{\Omega_r\times\mathbb{T}^n} |u_{r, \Theta_k}|^p d xdy \right|<\frac{m_{r, 1}\left(\Theta_k\right)}{2} \text { for any } r \geq \bar{r}_k .
$$
In view of \eqref{eq3.21} and $I_r\left(u_{r, \Theta_k}\right)=m_{r, 1}\left(\Theta_k\right)$, we further have
\begin{equation}\label{eq3.23}
\int_{\Omega_r\times\mathbb{T}^n}\left|\Delta u_{r, \Theta_k}\right|^2 d xdy+\int_{\Omega_r\times\mathbb{T}^n} V  u_{r, \Theta_k}^2 d xdy \geq \frac{m_{r, 1}\left(\Theta_k\right)}{2} \text { for any large } k \text { and } r \geq \bar{r}_k .
\end{equation}
It follows from \eqref{eq3.20}, \eqref{eq3.22} and \eqref{eq3.23} that there exists $r_k \geq \bar{r}_k$ with $r_k \rightarrow \infty$ as $k \rightarrow \infty$ such that
\begin{equation}\label{eq3.24}
   \lim _{k \rightarrow \infty} \max _{\Omega_{r_k}\times\mathbb{T}^n} u_{r_k, \Theta_k}=0,
\end{equation}
\begin{equation}\label{eq3.25}
   \int_{\Omega_{r_k}\times\mathbb{T}^n}\left|u_{r_k, \Theta_k}\right|^s d xdy \leq\left|\max _{\Omega_{r_k}\times\mathbb{T}^n} u_{r_k, \Theta_k}\right|^{s-2} \Theta_k \rightarrow 0 \text { as } k \rightarrow \infty \text { for any } s>2
\end{equation}
and
\begin{equation}\label{eq3.26}
\int_{\Omega_{r_k}\times\mathbb{T}^n}\left|\Delta u_{r_k, \Theta_k}\right|^2 d xdy+\int_{\Omega_{r_k}\times\mathbb{T}^n} V u_{r_k, \Theta_k}^2 d xdy \rightarrow \infty \quad \text { as } k \rightarrow \infty .
\end{equation}
By \eqref{eq1.1}, \eqref{eq3.25} and \eqref{eq3.26}, we have
\begin{equation}\label{eq3.27}
\lambda_{r_k,\Theta_k} \rightarrow-\infty \quad \text { as } k \rightarrow \infty .
\end{equation}
Now \eqref{eq1.1} implies
\begin{eqnarray*}
\Delta^2 u_{r_k, \Theta_k}+Vu_{r_k, \Theta_k}+\lambda_{r_k,\Theta_k} u_{r_k, \Theta_k}=|u_{r_k, \Theta_k}|^{q-2}u_{r_k, \Theta_k}+\mu|u_{r_k, \Theta_k}|^{p-2}u_{r_k, \Theta_k},
\end{eqnarray*}
so
$$
\Delta^2 u_{r_k, \Theta_k}+\left(\|V\|_{\infty}+\frac{\lambda_{r_k, \Theta_k}}{2}\right) u_{r_k, \Theta_k} \geq-\frac{\lambda_{r_k, \Theta_k}}{2}u_{r_k, \Theta_k}+\left|u_{r_k, \Theta_k}\right|^{q-2}u_{r_k, \Theta_k}+\mu|u_{r_k, \Theta_k}|^{p-2}u_{r_k, \Theta_k}.
$$
Using \eqref{eq3.27} and \eqref{eq3.24}, it follows that
$$
\Delta^2 u_{r_k, \Theta_k}+\left(\|V\|_{\infty}+\frac{\lambda_{r_k, \Theta_k}}{2}\right) u_{r_k, \Theta_k} \geq 0
$$
for large $k$. Let $\theta_{r_k}$ be the principal eigenvalue of $\Delta^2$ with Dirichlet boundary condition in $\Omega_{r_k}$, and $v_{r_k}>0$ be the corresponding normalized eigenfunction. It follows that
$$
\left(\theta_{r_k}+\|V\|_{\infty}+\frac{\lambda_{r_k, \Theta_k}}{2}\right) \int_{\Omega_{r_k}\times\mathbb{T}^n} u_{r_k, \Theta_k} v_{r_k} d xdy \geq 0.
$$
Since $\int_{\Omega_{r_k}\times\mathbb{T}^n} u_{r_k, \Theta_k} v_{r_k} d xdy>0$, we have
$$
\theta_{r_k}+\|V\|_{\infty}+\frac{\lambda_{r_k, \Theta_k}}{2} \geq 0,
$$
which contradicts \eqref{eq3.27} for large $k$. Hence the claim holds, that is, there exists $\Theta_1>0$ such that
\begin{equation}\label{eqq3.28}
  Q(\Theta)=\liminf\limits_{r \rightarrow \infty} \max\limits_{\Omega_r\times\mathbb{T}^n} u_{r, \Theta}>0
\end{equation}
for any $0<\Theta<\Theta_1$.

We consider $H^2(\Omega_r\times\mathbb{T}^n)$ as a subspace of $H^2(\mathbb{R}^d\times\mathbb{T}^n)$ for any $r>0$. It follows from Lemma \ref{L3.3} that the set of solutions $\left\{u_{r, \Theta}: r>r_\Theta\right\}$ established in Lemma \ref{L3.4} is bounded in $H^2(\mathbb{R}^d\times\mathbb{T}^n)$, so there exist $u_\Theta \in H^2(\mathbb{R}^d\times\mathbb{T}^n)$ and $\lambda_\Theta \in \mathbb{R}$ such that up to a subsequence:
\begin{eqnarray*}
  \lambda_{r, \Theta} &\rightarrow& \lambda_\Theta,\\
u_{r, \Theta} &\rightharpoonup& u_\Theta \ \text {in}\ H^2(\mathbb{R}^d\times\mathbb{T}^n),\\
u_{r, \Theta} &\rightarrow& u_\Theta \ \text{in}\ L_{l o c}^k(\mathbb{R}^d\times\mathbb{T}^n) \ \text{for all}\ 2 \leq k<2^*,\\
u_{r, \Theta} &\rightarrow& u_\Theta \ \text {a.e. in}\ \mathbb{R}^d\times\mathbb{T}^n
\end{eqnarray*}
and $u_\Theta$ is a solution of the equation
\begin{equation*}
  \Delta^2 u+V(x,y)u+\lambda_\Theta u=|u|^{q-2} u+\mu |u|^{p-2} u \ \text {in}\ \mathbb{R}^d\times\mathbb{T}^n.
\end{equation*}
Hence,
\begin{eqnarray}\label{eq3.28}
&&\int_{\mathbb{R}^d\times\mathbb{T}^n}\left|\Delta u_\Theta\right|^2 d xdy+\int_{\mathbb{R}^d\times\mathbb{T}^n} V(x,y) u_\Theta^2 d xdy+\lambda_\Theta \int_{\mathbb{R}^d\times\mathbb{T}^n} u_\Theta^2 d xdy\nonumber\\
&=&\int_{\mathbb{R}^d\times\mathbb{T}^n}\left|u_\Theta\right|^q d xdy+\mu \int_{\mathbb{R}^d\times\mathbb{T}^n}|u_\Theta|^p d xdy
\end{eqnarray}
and the Pohozaev identity gives
\begin{eqnarray}\label{eq3.29}
& &\frac{d+n-4}{2(d+n)} \int_{\mathbb{R}^d\times\mathbb{T}^n}\left|\Delta u_\Theta\right|^2 d xdy+\frac{1}{2(d+n)} \int_{\mathbb{R}^d\times\mathbb{T}^n} \widetilde{V} u_\Theta^2dxdy+\frac{1}{2} \int_{\mathbb{R}^d\times\mathbb{T}^n} V(x,y) u_\Theta^2 d xdy\nonumber\\
&&+\frac{\lambda_\Theta}{2} \int_{\mathbb{R}^d\times\mathbb{T}^n} u_\Theta^2 d xdy \nonumber\\
& =&\frac{1}{q} \int_{\mathbb{R}^d\times\mathbb{T}^n}\left|u_\Theta\right|^q d xdy+\frac{\mu}{p} \int_{\mathbb{R}^d\times\mathbb{T}^n}|u_\Theta|^p d xdy .
\end{eqnarray}
It follows from \eqref{eq3.28}, \eqref{eq3.29}, $(f_2)$, the Gagliardo-Nirenberg inequality and the fact $\mu\leq0$ that
\begin{eqnarray*}
  & &\frac{2}{d+n} \int_{\mathbb{R}^d\times\mathbb{T}^n}\left|\Delta u_\Theta\right|^2 d xdy+\frac{1}{2(d+n)} \int_{\mathbb{R}^d\times\mathbb{T}^n} \widetilde{V}(x,y) u_\Theta^2dxdy  \nonumber\\
  &=&\left(\frac{1}{2}-\frac{1}{q}\right) \int_{\mathbb{R}^d\times\mathbb{T}^n}\left|u_\Theta\right|^q d xdy+\mu \int_{\mathbb{R}^d\times\mathbb{T}^n}\left(\frac{1}{2}-\frac{1}{p}\right)|u_\Theta|^p d xdy  \nonumber\\
  &\leq& \frac{C_{d,n, q}(q-2)}{2 q}\left(\int_{\mathbb{R}^d\times\mathbb{T}^n} u_\Theta^2 d xdy\right)^{\frac{2 q-(d+n)(q-2)}{4}}\left(\int_{\mathbb{R}^d\times\mathbb{T}^n}\left|\Delta u_\Theta\right|^2 d x\right)^{\frac{(d+n)(q-2)}{4}}.
\end{eqnarray*}
By using the H\"older inequality, we have
\begin{eqnarray*}
&&\left(\frac{1}{d+n}-\frac{\|\widetilde{V}_{+}\|_{\frac{d+n}{4}} S^{-1}}{2(d+n)}\right) \int_{\mathbb{R}^d\times\mathbb{T}^n}\left|\Delta u_\Theta\right|^2 d xdy\\
&\leq&\frac{1}{ N} \int_{\mathbb{R}^d\times\mathbb{T}^n}\left|\Delta u_\Theta\right|^2 d xdy+\frac{1}{2 N} \int_{\mathbb{R}^d\times\mathbb{T}^n} \widetilde{V}(x,y) u_\Theta^2dxdy.
\end{eqnarray*}
Therefore,
\begin{eqnarray*}
   &&\left(\frac{1}{d+n}-\frac{\|\widetilde{V}_{+}\|_{\frac{d+n}{4}} S^{-1}}{2(d+n)}\right) \int_{\mathbb{R}^d\times\mathbb{T}^n}\left|\Delta u_\Theta\right|^2 d xdy\nonumber\\
   &\leq&\frac{C_{d,n, q}(q-2)}{2 q}\left(\int_{\mathbb{R}^d\times\mathbb{T}^n} u_\Theta^2 d xdy\right)^{\frac{2 q-(d+n)(q-2)}{4}}\left(\int_{\mathbb{R}^d\times\mathbb{T}^n}\left|\Delta u_\Theta\right|^2 d x\right)^{\frac{(d+n)(q-2)}{4}}.
\end{eqnarray*}
If $u_\Theta \neq 0$, Using $\|\widetilde{V}_{+}\|_{\frac{d+n}{4}} < 2 S$, we obtain that
\begin{equation}\label{eq3.30}
  \int_{\mathbb{R}^d\times\mathbb{T}^n}\left|\Delta u_\Theta\right|^2 d xdy \geq\left[\frac{q\left(2-\|\widetilde{V}_{+}\|_{\frac{d+n}{4}} S^{-1}\right)}{(d+n)C_{d,n, q}(q-2)}\right]^{\frac{4}{(d+n)(q-2)-4}} \Theta^{\frac{q(d+n-2)-2(d+n)}{(d+n)(q-2)-4}} .
\end{equation}
Next, it follows from \eqref{eq3.28}, \eqref{eq3.29}, \eqref{eq3.30}, $(f_2)$ and $2+\frac{4}{N}<q<4^*$ that
\begin{eqnarray*}
&& \left(\frac{1}{q}-\frac{1}{ 2}\right)\lambda_\Theta \int_{\mathbb{R}^d\times\mathbb{T}^n} u_\Theta^2 d xdy\\
&=& \left(\frac{d+n-4}{2(d+n)}-\frac{1}{q}\right)\int_{\mathbb{R}^d\times\mathbb{T}^n}\left|\Delta u_\Theta\right|^2 d xdy+\frac{1}{2(d+n)} \int_{\mathbb{R}^d\times\mathbb{T}^n} \widetilde{V}(x) u_\Theta^2 d xdy \\
&&+\left(\frac{1}{2  }-\frac{1}{q}\right) \int_{\mathbb{R}^d\times\mathbb{T}^n} V(x) u_\Theta^2 d xdy-\frac{(q-p)\mu}{pq} \int_{\mathbb{R}^d\times\mathbb{T}^n}|u_\Theta|^pd xdy \\
& \leq& \frac{(d+n-4) q-2(d+n)}{2q(d+n)} \int_{\mathbb{R}^d\times\mathbb{T}^n}\left|\Delta u_\Theta\right|^2 d xdy+\frac{\|\widetilde{V}\|_{\infty}}{2(d+n)} \Theta+\frac{(q-2)\|V\|_{\infty}}{2 q} \Theta \\
&&-\frac{\mu(q-p)}{q}C_{d,n,p}   \Theta^{ \frac{2 p-(d+n)(p-2)}{4}}\left(\int_{\mathbb{R}^d\times\mathbb{T}^n}\left|\Delta u_\Theta\right|^2 d xdy\right)^{\frac{(d+n)(p-2)}{4}} \\
& \rightarrow&-\infty \ \text { as } \Theta \rightarrow 0,
\end{eqnarray*}
since $\frac{(d+n-2) q-2(d+n)}{2q(d+n)}<0$. Therefore, if $u_\Theta \neq 0$ for $\Theta>0$ small there exists $\Theta_0>0$ such that $\lambda_\Theta>0$ for $0<\Theta<\Theta_0$.

In order to complete the proof, we consider the case that there is a sequence $\Theta_k \rightarrow 0$ such that $u_{\Theta_k}=0$ for any $k$. Assume without loss of generality that $u_\Theta=0$ for any $\Theta \in(0, \Theta_1)$. Let $(x_{r, \Theta},y_{r, \Theta}) \in \Omega_r\times\mathbb{T}^n$ be such that $u_{r, \Theta}\left(x_{r, \Theta},y_{r, \Theta}\right)=\max\limits_{\Omega_r\times\mathbb{T}^n} u_{r, \Theta}$. In view of \eqref{eqq3.28}, there holds $\left|x_{r, \Theta}\right| \rightarrow \infty$
as $r \rightarrow \infty$. Otherwise, there exists $(x_0,y_0) \in \Omega_r\times\mathbb{T}^n$ such that, up to a subsequence, $x_{r, \Theta} \rightarrow x_0$, $y_{r, \Theta} \rightarrow y_0$ and hence $u_\Theta(x_0,y_0) \geq d_\Theta>0$. This contradicts $u_\Theta=0$. We claim that $\operatorname{dist}((x_{r, \Theta},y_{r, \Theta}), \partial (\Omega_r\times\mathbb{T}^n)) \rightarrow \infty$ as $r \rightarrow \infty$. Arguing by contradiction we assume that $\liminf\limits_{r \rightarrow \infty} \operatorname{dist}((x_{r, \Theta},y_{r, \Theta}),  \partial (\Omega_r\times\mathbb{T}^n))=l<\infty$. It follows from \eqref{eqq3.28} that $l>0$. Let $w_r(x)=u_{r, \Theta}(x+x_{r, \Theta},y+y_{r, \Theta})$ for any $(x,y) \in \Sigma^r:=\{(x,y) \in \mathbb{R}^d\times\mathbb{T}^n: x+x_{r, \Theta} \in \Omega_r,y+y_{r, \Theta} \in \mathbb{T}^n\}$. Then $w_r$ is bounded in $H^2(\mathbb{R}^d\times\mathbb{T}^n)$, and there is $w \in H^2(\mathbb{R}^d\times\mathbb{T}^n)$ such that $w_r \rightharpoonup w$ as $r \rightarrow \infty$. By the regularity theory of elliptic partial equations and $\liminf\limits_{r \rightarrow \infty} u_{r, \Theta}\left(x_{r, \Theta},y_{r, \Theta}\right)>d_\Theta>0$, we infer that $w(0) \geq d_\Theta>0$. Assume without loss of the generality that, up to a subsequence,
$$
\lim\limits_{r \rightarrow \infty} \frac{x_{r, \Theta}}{\left|x_{r, \Theta}\right|}=e_1.
$$
Setting
$$
\Sigma=\left\{x \in \mathbb{R}^d\times\mathbb{T}^n: x \cdot e_1<l\right\}=\left\{x \in \mathbb{R}^d\times\mathbb{T}^n: x_1<l\right\},
$$
we have $\phi(x-x_{r, \Theta},y-y_{r, \Theta}) \in C_c^{\infty}(\Omega_r\times\mathbb{T}^n)$ for any $\phi \in C_c^{\infty}(\Sigma)$ and $r$ large enough. It then follows that
\begin{eqnarray}\label{eq3.32}
&&\int_{\Omega_r\times\mathbb{T}^n} \Delta u_{r, \Theta} \Delta \phi\left(x-x_{r, \Theta},y-y_{r, \Theta}\right) d xdy+\int_{\Omega_r\times\mathbb{T}^n} V u_{r, \Theta} \phi\left(x-x_{r, \Theta},y-y_{r, \Theta}\right) d xdy\nonumber\\
&&+\lambda_{r, \Theta} \int_{\Omega_r\times\mathbb{T}^n} u_{r, \Theta} \phi\left(x-x_{r, \Theta},y-y_{r, \Theta}\right) d xdy\nonumber\\
&=&\int_{\Omega_r\times\mathbb{T}^n}\left|u_{r, \Theta}\right|^{q-2} u_{r, \Theta} \phi\left(x-x_{r, \Theta},y-y_{r, \Theta}\right) d xdy\nonumber\\
&&+\mu \int_{\Omega_r\times\mathbb{T}^n}\left|u_{r, \Theta}\right|^{p-2} u_{r, \Theta} \phi\left(x-x_{r, \Theta},y-y_{r, \Theta}\right)dxdy .
\end{eqnarray}
Since $\left|x_{r, \Theta}\right| \rightarrow \infty$ as $r \rightarrow \infty$, it holds
\begin{eqnarray}\label{eq3.33}
\left|\int_{\Omega_r\times\mathbb{T}^n} V u_{r, \Theta} \phi\left(x-x_{r, \Theta}\right) d xdy\right| & \leq& \int_{\text {Supp } \phi\times\mathbb{T}^n}\left|V\left(x+x_{r, \Theta}\right) w_r \phi\right| d x dy \nonumber\\
& \leq&\left\|w_r\right\|_{4^*}\|\phi\|_{4^*}\left(\int_{\text {Supp } \phi\times\mathbb{T}^n}\left|V\left(\cdot+x_{r, \Theta}\right)\right|^{\frac{d+n}{4}} d xdy\right)^{\frac{4}{d+n}} \nonumber\\
& \leq&\left\|w_r\right\|_{4^*}\|\phi\|_{4^*}\left(\int_{\mathbb{R}^d \backslash B_{\frac{|x_{r, \Theta}|}{2}}\times\mathbb{T}^n}|V|^{\frac{d+n}{4}} d xdy\right)^{\frac{4}{d+n}}\nonumber\\
& \rightarrow& 0 \text { as } r \rightarrow \infty .
\end{eqnarray}
Letting $r \rightarrow \infty$ in \eqref{eq3.32}, we obtain for $\phi \in C_c^{\infty}(\Sigma)$:
$$
\int_{\Sigma} \nabla w \cdot \nabla \phi d xdy+\lambda_\Theta \int_{\Sigma} w \phi d xdy=\int_{\Sigma}|w|^{q-2} w \phi d xdy+\mu \int_{\Sigma}f(w) \phi d xdy .
$$
Thus $w \in H_0^2(\Sigma)$ is a weak solution of the equation
\begin{equation}\label{eq3.34}
 \Delta^2 w+\lambda_\Theta w=|w|^{q-2} w+\mu|w|^{p-2} w \quad \text { in } \Sigma .
\end{equation}
Hence we obtain a nontrivial nonnegative solution of \eqref{eq3.34} on a half space which is impossible by the Liouville theorem (see \cite{MEPL1982}). This proves that dist $\left((x_{r, \Theta},y_{r, \Theta}), \partial (\Omega_r\times\mathbb{T}^n)\right) \rightarrow \infty$ as $r \rightarrow \infty$. A similar argument as above shows that \eqref{eq3.34} holds for $\Sigma=\mathbb{R}^d\times\mathbb{T}^n$. Now we argue as in the case $u_\Theta \neq 0$ above that there exists $\Theta_2$ such that $\lambda_\Theta>0$ for any $0<\Theta<\Theta_2$.

Setting $\bar{\Theta}=\min \left\{\Theta_0, \Theta_1, \Theta_2\right\}$, the proof is complete.
\end{proof}
\noindent\textbf{Proof of Theorem \ref{t1.1}} The proof is an immediate consequence of Lemmas \ref{L3.4}, \ref{L3.5} and \ref{L3.6}.

\section{Proof of Theorem \ref{t1.2}}

In this section, we assume that the assumptions of Theorem \ref{t1.2} hold.
Since $\mu>0$,
\begin{eqnarray*}
&&I_r(u)\\
&\geq&\frac{1-\left\|V_{-}\right\|_{\frac{d+n}{4}} S^{-1}}{2} \int_{\Omega_r\times\mathbb{T}^n}|\Delta u|^2 d xdy-\frac{C_{d,n, q} \Theta^{\frac{4q-(d+n)(q-2)}{8}}}{q}\left(\int_{\Omega_r\times\mathbb{T}^n}|\Delta u|^2 d xdy\right)^{\frac{(d+n)(q-2)}{8}}\\
&&-\mu C_{d,n, p} \Theta^{\frac{4 p-(d+n)(p-2)}{8}}\left(\int_{\Omega_r\times\mathbb{T}^n}|\Delta u|^2 d xdy\right)^{\frac{(d+n)(p-2)}{8}}\\
&=&h_1(t),
\end{eqnarray*}
where
\begin{eqnarray*}
h_1(t)&:=&\frac{1}{2}\left(1-\left\|V_{-}\right\|_{\frac{d+n}{4}} S^{-1}\right) t^4-\frac{C_{d,n, q} \Theta^{\frac{4 q-(d+n)(q-2)}{8}}}{q}t^{\frac{(d+n)(q-2)}{2}}\\
&&- \mu C_{d,n, p} \Theta^{\frac{4 p-(d+n)(p-2)}{8}}t^{\frac{(d+n)(p-2)}{2}}\\
&=&t^{\frac{(d+n)(p-2)}{2}}\left[\frac{1}{2}\left(1-\left\|V_{-}\right\|_{\frac{d+n}{4}} S^{-1}\right) t^{\frac{8-(d+n)(p-2)}{2}}-\frac{C_{d,n, q} \Theta^{\frac{4 q-(d+n)(q-2)}{8}}}{q}t^{\frac{(d+n)(q-p)}{2}}\right]\\
&&-\mu C_{d,n, p} \Theta^{\frac{4 p-(d+n)(p-2)}{8}}t^{\frac{(d+n)(p-2)}{2}}.
\end{eqnarray*}
Consider
$$
\psi(t):=\frac{1}{2}\left(1-\left\|V_{-}\right\|_{\frac{d+n}{4}} S^{-1}\right) t^{\frac{8-(d+n)(p-2)}{2}}-\frac{C_{d,n, q} \Theta^{\frac{4 q-(d+n)(q-2)}{8}}}{q}t^{\frac{(d+n)(q-p)}{2}}.
$$
Note that $\psi$ admits a unique maximum at
$$
\bar{t}=\left[\frac{q(8-(d+n)(p-2))\left(1-\|V_{-}\|_{\frac{d+n}{4}} S^{-1}\right)}{ 2(d+n)(q-p) C_{d,n, q}}\right]^{\frac{2}{(d+n)(q-2)-8}} \Theta^{\frac{(d+n)(q-2)-4 q}{4((d+n)(q-2)-8)}} .
$$
By a direct calculation, we obtain
\begin{eqnarray*}
&&\psi(\bar{t})\\
&=&\frac{1}{2}\left(1-\left\|V_{-}\right\|_{\frac{d+n}{4}} S^{-1}\right) \bar{t}^{\frac{8-(d+n)(p-2)}{2}}-\frac{C_{d,n, q} \Theta^{\frac{4 q-(d+n)(q-2)}{8}}}{q}\bar{t}^{\frac{(d+n)(q-p)}{2}}\\
&=&\frac{1}{2}\left(1-\left\|V_{-}\right\|_{\frac{d+n}{4}} S^{-1}\right)^{\frac{(d+n)(q-p)}{(d+n)(q-2)-8}}\left[\frac{q(8-(d+n)(p-2))}{ 2(d+n)(q-p) C_{d,n, q}}\right]^{\frac{8-(d+n)(p-2)}{(d+n)(q-2)-8}} \Theta^{\frac{[(d+n)(q-2)-4 q][8-(d+n)(p-2)]}{8((d+n)(q-2)-8)}}\\
&&-\frac{C_{d,n, q} \Theta^{\frac{4 q-(d+n)(q-2)}{8}}}{q}\left[\frac{q(8-(d+n)(p-2))\left(1-\|V_{-}\|_{\frac{d+n}{4}} S^{-1}\right)}{ 2(d+n)(q-p) C_{d,n, q}}\right]^{\frac{(d+n)(q-p)}{(d+n)(q-2)-8}} \Theta^{\frac{[(d+n)(q-2)-4 q](d+n)(q-p)}{8((d+n)(q-2)-8)}}\\
&=&\frac{1}{2}\left(1-\left\|V_{-}\right\|_{\frac{d+n}{4}} S^{-1}\right)^{\frac{(d+n)(q-p)}{(d+n)(q-2)-8}}\left[\frac{q(8-(d+n)(p-2))}{ 2(d+n)(q-p) C_{d,n, q}}\right]^{\frac{8-(d+n)(p-2)}{(d+n)(q-2)-8}} \Theta^{\frac{[(d+n)(q-2)-4 q][8-(d+n)(p-2)]}{8((d+n)(q-2)-8)}}\\
&&-\frac{C_{d,n, q} }{q}\left[\frac{q(8-(d+n)(p-2))\left(1-\|V_{-}\|_{\frac{d+n}{4}} S^{-1}\right)}{ 2(d+n)(q-p) C_{d,n, q}}\right]^{\frac{(d+n)(q-p)}{(d+n)(q-2)-8}} \Theta^{\frac{[(d+n)(q-2)-2 q][8-(d+n)(p-2)]}{8((d+n)(q-2)-8)}}.
\end{eqnarray*}
Hence,
$$
\psi(\bar{t})>\mu C_{N, p}\Theta^{\frac{4 p-(d+n)(p-2)}{8}}
$$
as long as
\begin{eqnarray*}
&&\mu C_{d,n, p}\Theta^{\frac{4 p-(d+n)(p-2)}{8}-\frac{[(d+n)(q-2)-4 q][8-(d+n)(p-2)]}{8((d+n)(q-2)-8)}}\\
&<&\frac{1}{2}\left(1-\left\|V_{-}\right\|_{\frac{d+n}{4}} S^{-1}\right)^{\frac{(d+n)(q-p)}{(d+n)(q-2)-8}}\left[\frac{q(8-(d+n)(p-2))}{ 2(d+n)(q-p) C_{d,n, q}}\right]^{\frac{8-(d+n)(p-2)}{(d+n)(q-2)-8}} \\
&&-\frac{C_{d,n, q} }{q}\left[\frac{q(8-(d+n)(p-2))\left(1-\|V_{-}\|_{\frac{d+n}{4}} S^{-1}\right)}{ 2(d+n)(q-p) C_{d,n, q}}\right]^{\frac{(d+n)(q-p)}{(d+n)(q-2)-8}}\\
&=&\left[\frac{1-\left\|V_{-}\right\|_{\frac{d+n}{4}} S^{-1}}{2(d+n)(q-p)}\right]^{\frac{(d+n)(q-p)}{(d+n)(q-2)-8}}\left[\frac{q(8-(d+n)(p-2))}{C_{d,n, q}}\right]^{\frac{8-(d+n)(p-2)}{(d+n)(q-2)-8}}\left[(d+n)(q-2)-8\right],
\end{eqnarray*}
that is,
\begin{equation*}
  \Theta<\left[\frac{1-\left\|V_{-}\right\|_{\frac{d+n}{4}} S^{-1}}{2(d+n)(q-p)}\right]^{\frac{d+n}{4}}\left[\frac{q(8-(d+n)(p-2))}{C_{d,n, q}}\right]^{\frac{8-(d+n)(p-2)}{4 (q-p)}}\left[\frac{N(q-2)-4}{\mu C_{d,n, p}}\right]^{\frac{(d+n)(q-2)-8}{4 (q-p)}}.
\end{equation*}
Hence, we take
\begin{equation*}
  \Theta_V=\left[\frac{1-\left\|V_{-}\right\|_{\frac{d+n}{4}} S^{-1}}{2(d+n)(q-p)}\right]^{\frac{d+n}{4}}\left[\frac{q(8-(d+n)(p-2))}{C_{d,n, q}}\right]^{\frac{8-(d+n)(p-2)}{4 (q-p)}}\left[\frac{N(q-2)-4}{\mu C_{d,n, p}}\right]^{\frac{(d+n)(q-2)-8}{4 (q-p)}}.
\end{equation*}
Now, let $0<\Theta<\Theta_V$ be fixed, we obtain
\begin{equation}\label{eq4.1}
  \psi(\bar{t})>\mu C_{d,n, p}\Theta^{\frac{4 p-(d+n)(p-2)}{8}}
\end{equation}
and $h_1(\bar{t})>0$. In view of $2<p<2+\frac{8}{d+n}<q<4^*$ and \eqref{eq4.1}, there exist $0<R_1<T_\Theta<R_2$ such that $h_1(t)<0$ for $0<t<R_1$ and for $t>R_2, h_1(t)>0$ for $R_1<t<R_2$, and $h_1\left(T_\Theta\right)=\max\limits_{t \in \mathbb{R}^{+}} h_1(t)>0$.

Define
$$
\mathcal{V}_{r, \Theta}=\left\{u \in S_{r, \Theta}:\|\Delta u\|_2^2 \leq T_\Theta^2\right\} .
$$
Let $\theta$ be the principal eigenvalue of operator $\Delta^2$ with Dirichlet boundary condition in $\Omega$, and let $|\Omega|$ be the volume of $\Omega$.
\begin{lemma}\label{L4.1}

(i) If $r<\frac{\sqrt{C \Theta}}{T_\Theta}$, then $\mathcal{V}_{r, \Theta}=\emptyset$.

(ii) If
$$
r> \max \left\{\frac{\sqrt{C\Theta}}{T_\Theta},\left(\frac{ \theta\left(1+\|V\|_{\frac{d+n}{4}} S^{-1}\right)}{2\mu} \Theta^{\frac{2-p}{2}}|\Omega|^{\frac{p-2}{2}}\right)^{\frac{2}{(d+n)(p-2)+8}}\right\}
$$
then $\mathcal{V}_{r, \Theta} \neq \emptyset$ and
$$
e_{r, \Theta}:=\inf _{u \in \mathcal{V}_{r, \Theta}} I_r(u)<0
$$
is attained at some interior point $u_r$ of $\mathcal{V}_{r, \Theta}$. As a consequence, there exists a Lagrange multiplier $\lambda_r \in \mathbb{R}$ such that $\left(\lambda_r, u_r\right)$ is a solution of \eqref{eq1.1}. Moreover $\liminf\limits_{r \rightarrow \infty} \lambda_r>0$ holds true.
\end{lemma}
\begin{proof}
(i) The embedding inequality implies there exists a positive constant $C$(only depend on $\Omega$) such that
$$
\int_{\Omega_r\times\mathbb{T}^n}|\Delta u|^2 d xdy=\frac{1}{r^2}\int_{\Omega_r\times\mathbb{T}^n}|\Delta u|^2 d xdy \geq \frac{C}{r^2}\int_{\Omega_r\times\mathbb{T}^n}| u|^2 d xdy  =\frac{C \Theta}{r^2}
$$
for any $u \in S_{r, \Theta}$. Since $T_\Theta$ is independent of $r$, there holds $\mathcal{V}_{r, \Theta}=\emptyset$ if and only if $r<\frac{\sqrt{C \Theta}}{T_\Theta}$.

(ii) Let $v_1 \in S_{1, \Theta}$ be the positive normalized eigenfunction corresponding to $\theta$. Setting
\begin{equation}\label{eq4.2}
  r_\Theta=\max \left\{\frac{\sqrt{C\Theta}}{T_\Theta},\left[\frac{ \theta\left(1+\|V\|_{\frac{d+n}{4}} S^{-1}\right)}{2\mu} \Theta^{\frac{2-p}{2}}|\Omega_r\times\mathbb{T}^n|^{\frac{p-2}{2}}\right]^{\frac{2}{(d+n)(p-2)+8}}\right\}.
\end{equation}
Now, we construct for $r>r_\Theta$ a function $u_r \in S_{r, \Theta}$ such that $u_r \in \mathcal{V}_{r, \Theta}$ and $I_r\left(u_r\right)<0$. Clearly,
$$
\int_{\Omega_r\times\mathbb{T}^n}\left|\nabla v_1\right|^2 d xdy=\theta \Theta,\ \Theta=\int_{\Omega_r\times\mathbb{T}^n}\left|v_1\right|^2 d xdy \leq\left(\int_{\Omega_r\times\mathbb{T}^n}\left|v_1\right|^{p} d xdy\right)^{\frac{2}{p}}|\Omega|^{\frac{p-2}{p}} .
$$
Define $u_r \in S_{r, \Theta}$ by $u_r(x,y)=r^{-\frac{d+n}{2}} v_1\left(r^{-1} x,r^{-1} y\right)$ for $(x,y) \in \Omega_r\times\mathbb{T}^n$. Then
\begin{equation}\label{eq4.3}
  \int_{\Omega_r\times\mathbb{T}^n}\left|\Delta u_r\right|^2 d xdy=r^{-4} \theta \Theta \quad \text { and } \quad \int_{\Omega_r\times\mathbb{T}^n}\left|u_r\right|^{p} d xdy \geq r^{\frac{(d+n)(2-p)}{2}} \Theta^{\frac{p}{2}}|\Omega_r\times\mathbb{T}^n|^{\frac{2-p}{2}}.
\end{equation}
By \eqref{eq4.2}, \eqref{eq4.3}, $2<p<2+\frac{8}{d+n}$ and a direct calculation we have $u_r \in \mathcal{V}_{r, \Theta}$ and
\begin{eqnarray*}
I_r\left(u_r\right)
& \leq&\frac{1}{2}\left(1+\|V\|_{\frac{d+n}{4}} S^{-1}\right) r^{-4} \theta \Theta-\mu r^{\frac{(d+n)(2-p)}{2}} \Theta^{\frac{p}{2}}|\Omega|^{\frac{2-p}{2}} \\
&<& 0 .
\end{eqnarray*}
It then follows from the Gagliardo-Nirenberg inequality that
\begin{eqnarray}\label{eq4.5}
&&I_r\left(u_r\right)\nonumber\\
&\geq& \frac{1-\|V_{-}\|_{\frac{d+n}{4}} S^{-1}}{2}  \int_{\Omega_r\times\mathbb{T}^n}|\Delta u|^2 d xdy-C_{d,n, p} \mu \Theta^{\frac{4 p-(d+n)(p-2)}{8}}\left(\int_{\Omega_r\times\mathbb{T}^n}|\Delta u|^2 d xdy\right)^{\frac{(d+n)(p-2)}{8}} \nonumber\\
& &-\frac{C_{d,n, q}}{q} \Theta^{\frac{4 q-(d+n)(q-2)}{8}}\left(\int_{\Omega_r\times\mathbb{T}^n}|\Delta u|^2 d xdy\right)^{\frac{(d+n)(q-2)}{8}} .
\end{eqnarray}
As a consequence $I_r$ is bounded from below in $\mathcal{V}_{r, \Theta}$. By the Ekeland principle there exists a sequence $\left\{u_{n, r}\right\} \subset \mathcal{V}_{r, \Theta}$ such that
$$
I_r(u_{n, r}) \rightarrow \inf\limits_{u \in \mathcal{V}_{r, \Theta}} I_r(u),\  I_r^{\prime}(u_{n, r})|_{T_{u_n, r} S_{r, \Theta}} \rightarrow 0 \ \text {as}\ n \rightarrow \infty
$$
Consequently there exists $u_r \in H_0^2(\Omega_r\times\mathbb{T}^n)$ such that $u_{n, r} \rightharpoonup u_r$ in $H_0^2(\Omega_r\times\mathbb{T}^n)$ and
\begin{equation*}
  u_{n, r} \rightarrow u_r \ \text{in}\ L^k(\Omega_r\times\mathbb{T}^n) \ \text{for all}\ 2 \leq k<4^*.
\end{equation*}
Moreover,
$
\left\|\Delta u_r\right\|_2^2 \leq \liminf\limits_{n \rightarrow \infty}\left\|\Delta u_{n, r}\right\|_2^2 \leq T_\Theta^2,
$
that is, $u_r \in \mathcal{V}_{r, \Theta}$. Note that
$$
\int_{\Omega_r\times\mathbb{T}^n} V u_{n, r}^2 d x dy\rightarrow \int_{\Omega_r\times\mathbb{T}^n} V u_r^2 d xdy \ \text {as}\ n \rightarrow \infty,
$$
hence
$$
e_{r, \Theta} \leq I_r(u_r) \leq \liminf\limits_{n \rightarrow \infty} I_r(u_{n, r})=e_{r, \Theta}.
$$
It follows that $u_{n, r} \rightarrow u_r$ in $H_0^2\left(\Omega_r\right)$, so $I_r(u_r)<0$. Therefore $u$ is an interior point of $\mathcal{V}_{r, \Theta}$ because $I_r(u) \geq h_1(T_\Theta)>0$ for any $u \in \partial \mathcal{V}_{r, \Theta}$ by \eqref{eq4.5}. The Lagrange multiplier theorem implies that there exists $\lambda_r \in \mathbb{R}$ such that $(\lambda_r, u_r)$ is a solution of \eqref{eq1.1}. Moreover,
\begin{eqnarray}\label{eq4.6}
\lambda_r \Theta & =&\int_{\Omega_r\times\mathbb{T}^n}\left|u_r\right|^q d xdy+\mu\int_{\Omega_r\times\mathbb{T}^n}\left|u_r\right|^p d xdy-\int_{\Omega_r\times\mathbb{T}^n}\left|\Delta u_r\right|^2 d xdy-\int_{\Omega_r\times\mathbb{T}^n} V u_r^2 d xdy \nonumber\\
& =&\frac{q-2}{q}\int_{\Omega_r\times\mathbb{T}^n}\left|u_r\right|^q d xdy+\mu \int_{\Omega_r\times\mathbb{T}^n}\left|u_r\right|^p d xdy -\frac{2\mu}{p}\int_{\Omega_r\times\mathbb{T}^n}\left|u_r\right|^p d xdy-2 I_r(u_r) \nonumber\\
& >&-2 I_r(u_r)=-2 e_{r, \Theta} .
\end{eqnarray}
It follows from the definition of $e_{r, \Theta}$ that $e_{r, \Theta}$ is nonincreasing with respect to $r$. Hence, $e_{r, \Theta} \leq e_{r_\Theta, \Theta}<0$ for any $r>r_\Theta$ and $0<\Theta<\Theta_V$. In view of \eqref{eq4.6}, we have $\liminf\limits_{r \rightarrow \infty} \lambda_r>0$.
\end{proof}

\noindent\textbf{Proof of Theorem \ref{t1.2}} The proof is a direct consequence of Lemma \ref{L4.1} and Lemma \ref{L3.5}.

\section{Proof of Theorem \ref{t1.3}}

In this subsection we assume that the assumptions of Theorem \ref{t1.3} hold. For $s \in\left[\frac{1}{2}, 1\right]$, $\mu>0$, we define the functional $J_{r, s}: S_{r, \Theta} \rightarrow \mathbb{R}$ by
$$
J_{r, s}(u)=\frac{1}{2} \int_{\Omega_r\times\mathbb{T}^n}(|\Delta u|^2 d xdy +V u^2) d xdy-s\left(\frac{1}{q} \int_{\Omega_r\times\mathbb{T}^n}|u|^q d xdy+\frac{\mu}{p} \int_{\Omega_r\times\mathbb{T}^n}|u|^p d xdy\right) .
$$
Note that if $u \in S_{r, \Theta}$ is a critical point of $J_{r, s}$ then there exists $\lambda \in \mathbb{R}$ such that $(\lambda, u)$ is a solution of the problem
\begin{equation}\label{eq5.1}
 \left\{\aligned
& \Delta^2 u+V u+\lambda u=s|u|^{q-2}u+s\mu |u|^{p-2}u, &(x,y)\in \Omega_r\times\mathbb{T}^n, \\
& \int_{\Omega_r\times\mathbb{T}^n} |u|^2dxdy=\Theta,u\in H_0^2(\Omega_r\times\mathbb{T}^n), &(x,y)\in \Omega_r\times\mathbb{T}^n.
\endaligned
\right.
\end{equation}
\begin{lemma}\label{L5.1}
For $0<\Theta<\widetilde{\Theta}_V$ where $\widetilde{\Theta}_V$ is defined in Theorem \ref{t1.3}, there exist $\widetilde{r}_\Theta>0$ and $u^0, u^1 \in S_{r_\Theta, \Theta}$ such that

(i) For $r>\widetilde{r}_\Theta$ and $s \in\left[\frac{1}{2}, 1\right]$ we have $J_{r, s}\left(u^1\right) \leq 0$ and
$$
J_{r, s}\left(u^0\right)<\frac{((d+n)(q-2)-8)}{8}\left[\frac{4\left(1-\left\|V_{-}\right\|_{\frac{d+n}{4}} S^{-1}\right)}{(d+n)(q-2)}\right]^{\frac{(d+n)(q-2)}{(d+n)(q-2)-8}} A^{\frac{8}{8-(d+n)(q-2)}} \Theta^{\frac{(d+n)(q-2)-4 q}{(d+n)(q-2)-8}},
$$
where
$$
A=\left(\frac{C_{d,n,q}(q-2)((d+n)(q-2)-8)}{q(p-2)(8-(d+n)(p-2))} +\frac{C_{d,n, q}}{q}\right).
$$
Moreover,
$$
\left\|\Delta u^0\right\|_2^2<\left[\frac{4\left(1-\left\|V_{-}\right\|_{\frac{d+n}{4}} S^{-1}\right)}{(d+n)(q-2) A}\right]^{\frac{8}{(d+n)(q-2)-8}} \Theta^{\frac{(d+n)(q-2)-4 q}{(d+n)(q-2)-8}}
$$
and
$$
\left\|\Delta u^1\right\|_2^2>\left[\frac{4\left(1-\left\|V_{-}\right\|_{\frac{d+n}{4}} S^{-1}\right)}{(d+n)(q-2) A}\right]^{\frac{8}{(d+n)(q-2)-8}} \Theta^{\frac{(d+n)(q-2)-4 q}{(d+n)(q-2)-8}}.
$$

(ii) If $u \in S_{r, \Theta}$ satisfies
$$
\|\Delta u\|_2^2=\left[\frac{4\left(1-\left\|V_{-}\right\|_{\frac{d+n}{4}} S^{-1}\right)}{(d+n)(q-2) A}\right]^{\frac{8}{(d+n)(q-2)-8}} \Theta^{\frac{(d+n)(q-2)-4 q}{(d+n)(q-2)-8}},
$$
then there holds
$$
J_{r, s}(u) \geq \frac{((d+n)(q-2)-8)}{8}\left[\frac{4\left(1-\left\|V_{-}\right\|_{\frac{d+n}{4}} S^{-1}\right)}{(d+n)(q-2)}\right]^{\frac{(d+n)(q-2)}{(d+n)(q-2)-8}} A^{\frac{8}{8-(d+n)(q-2)}} \Theta^{\frac{(d+n)(q-2)-4 q}{(d+n)(q-2)-8}}.
$$

(iii) Let
$$
m_{r, s}(\Theta)=\inf _{\gamma \in \Gamma_{r, \Theta}} \sup\limits_{t \in[0,1]} J_{r, s}(\gamma(t)),
$$
where
$$
\Gamma_{r, \Theta}=\left\{\gamma \in C\left([0,1], S_{r, \Theta}\right): \gamma(0)=u^0, \gamma(1)=u^1\right\}.
$$
Then
$$
m_{r, s}(\Theta) \geq \frac{((d+n)(q-2)-8)}{8}\left[\frac{4\left(1-\left\|V_{-}\right\|_{\frac{d+n}{4}} S^{-1}\right)}{(d+n)(q-2)}\right]^{\frac{(d+n)(q-2)}{(d+n)(q-2)-8}} A^{\frac{8}{8-(d+n)(q-2)}} \Theta^{\frac{(d+n)(q-2)-4 q}{(d+n)(q-2)-8}}
$$
and
\begin{eqnarray*}
  m_{r, s}(\Theta) &\leq& \frac{(d+n)(q-2)-4}{2}\left(\frac{\theta\left(1+\|V\|_{\frac{d+n}{4}} S^{-1}\right)}{(d+n)(q-2)}\right)^{\frac{(d+n)(q-2)}{(d+n)(q-2)-4}}(4 q)^{\frac{4}{(d+n)(q-2)-4}}\\
 && \cdot|\Omega\times\mathbb{T}^n|^{\frac{2(q-2)}{(d+n)(q-2)-4}} \Theta^{\frac{(d+n)(q-2)-2 q}{(d+n)(q-2)-4}}.
\end{eqnarray*}
where $\theta$ is the principal eigenvalue of $\Delta^2$ with Dirichlet boundary condition in $\Omega\times\mathbb{T}^n$.
\end{lemma}
\begin{proof}
Let $v_1 \in S_{1, \Theta}$ be the positive normalized eigenfunction of $\Delta^2$ with Dirichlet boundary condition in $\Omega\times\mathbb{T}^n$ associated to $\theta$, then we have
\begin{equation}\label{eq5.2}
  \int_{\Omega\times\mathbb{T}^n}\left|\Delta v_1\right|^2 d xdy=\theta \Theta.
\end{equation}
By the H\"older inequality, we know
\begin{equation}\label{eq5.3}
  \int_{\Omega\times\mathbb{T}^n}|v_1(x)|^{p} dxdy\geq\Theta^{\frac{p}{2}}\cdot|\Omega\times\mathbb{T}^n|^{\frac{2-p}{2}}.
\end{equation}
Setting $v_t(x)=t^{\frac{d+n}{2}} v_1(t x,ty)$ for $(x,y) \in B_{\frac{1}{t}}$ and $t>0$. Using  \eqref{eq5.2}, \eqref{eq5.3} and $\frac{1}{2}\leq s\leq1$, we get
\begin{eqnarray}\label{eq5.4}
J_{\frac{1}{t}, s}\left(v_t\right)
&\leq & \frac{1+\|V\|_{\frac{d+n}{4}} S^{-1}}{2}  t^4 \theta\Theta- \frac{\mu}{2}  t^{\frac{(d+n)(p-2)}{2}}   \int_{\Omega\times\mathbb{T}^n}|v_1|^{p} d xdy \nonumber\\
&&-\frac{1}{2q} t^{\frac{(d+n)(q-2)}{2}} \Theta^{\frac{q}{2}}\cdot|\Omega\times\mathbb{T}^n|^{\frac{2-q}{2}}\nonumber\\
&\leq: & h_2(t) .
\end{eqnarray}
where
\begin{equation*}
  h_2(t)=\frac{1}{2}\left(1+\|V\|_{\frac{d+n}{4}} S^{-1}\right) t^4 \theta\Theta-\frac{1}{2q} t^{\frac{(d+n)(q-2)}{2}} \Theta^{\frac{q}{2}}\cdot|\Omega\times\mathbb{T}^n|^{\frac{2-q}{2}}
\end{equation*}
A simple computation shows that $h_2(t_0)=0$ for
$$
t_0:=\left[\left(1+\|V\|_{\frac{d+n}{2}} S^{-1}\right) q \theta \Theta^{\frac{2-q}{2}}|\Omega\times\mathbb{T}^n|^{\frac{q-2}{2}}\right]^{\frac{2}{(d+n)(q-2)-8}}
$$
and $h_2(t)<0$ for any $t>t_0, h_2(t)>0$ for any $0<t<t_0$. Moreover, $h_2(t)$ achieves its maximum at
$$
t_\Theta=\left[\frac{4 q\left(1+\|V\|_{\frac{d+n}{2}} S^{-1}\right) \theta}{(d+n)(q-2)} \Theta^{\frac{2-q}{2}}|\Omega\times\mathbb{T}^n|^{\frac{q-2}{2}}\right]^{\frac{2}{(d+n)(q-2)-8}} .
$$
This implies
\begin{equation}\label{eq5.5}
  J_{r, s}(v_{t_0})=J_{\frac{1}{t_0}, s}(v_{t_0}) \leq h_2(t_0)=0
\end{equation}
for any $r \geq \frac{1}{t_0}$ and $s \in\left[\frac{1}{2}, 1\right]$. There exists $0<t_1<t_\Theta$ such that for any $t \in\left[0, t_1\right]$,
\begin{equation}\label{eq5.6}
h_2(t)<\frac{((d+n)(q-2)-8)}{8}\left[\frac{4\left(1-\left\|V_{-}\right\|_{\frac{d+n}{4}} S^{-1}\right)}{(d+n)(q-2)}\right]^{\frac{(d+n)(q-2)}{(d+n)(q-2)-8}} A^{\frac{8}{8-(d+n)(q-2)}} \Theta^{\frac{(d+n)(q-2)-4 q}{(d+n)(q-2)-8}}  .
\end{equation}
On the other hand, it follows from  the Gagliardo-Nirenberg inequality and the H\"older inequality that
\begin{eqnarray}\label{eq5.7}
&&J_{r, s}(u)\nonumber\\
&\geq&\frac{1-\left\|V_{-}\right\|_{\frac{d+n}{4}} S^{-1}}{2}  \int_{\Omega_r\times\mathbb{T}^n}|\Delta u|^2 d xdy-\frac{C_{d,n, q} \Theta^{\frac{4 q-(d+n)(q-2)}{8}}}{q}\left(\int_{\Omega_r\times\mathbb{T}^n}|\Delta u|^2 d xdy\right)^{\frac{(d+n)(q-2)}{8}}\nonumber\\
&&-\mu C_{d,n, p} \Theta^{\frac{4 p-(d+n)(p-2)}{8}}\left(\int_{\Omega_r\times\mathbb{T}^n}|\Delta u|^2 d xdy\right)^{\frac{(d+n)(p-2)}{8}}.
\end{eqnarray}
Define
\begin{eqnarray*}
g_1(t)&:=&\frac{1-\left\|V_{-}\right\|_{\frac{d+n}{4}} S^{-1}}{2} t-\frac{C_{d,n, q} \Theta^{\frac{4 q-(d+n)(q-2)}{8}}}{q} t^{\frac{(d+n)(q-2)}{8}}-\mu C_{d,n, p} \Theta^{\frac{4 p-(d+n)(p-2)}{8}}t^{\frac{(d+n)(p-2)}{8}}\\
&=&t^{\frac{(d+n)(p-2)}{8}}\left[\frac{1}{2}\left(1-\left\|V_{-}\right\|_{\frac{d+n}{2}} S^{-1}\right) t^{\frac{8-(d+n)(p-2)}{8}}-\frac{C_{d,n, q} \Theta^{\frac{4 q-(d+n)(q-2)}{8}}}{q}t^{\frac{(d+n)(q-p)}{8}}\right]\\
&&-\mu C_{d,n, p} \Theta^{\frac{4 p-(d+n)(p-2)}{8}}t^{\frac{(d+n)(p-2)}{8}}
\end{eqnarray*}
In view of $2<p<2+\frac{8}{d+n}<q<4^*$ and the definition of $\widetilde{\Theta}_V$, there exist $0<l_1<l_M<l_2$ such that $ g_1(t)<0$ for any $0<t<l_1$ and $t>l_2, g_1(t)>0$ for $l_1<t<l_2$ and $g_1\left(l_M\right)=\max\limits_{t \in \mathbb{R}^{+}} g_1(t)>0$. Let
$$
t_2=\left(\frac{\mu q C_{d,n, p}(p-2)(8-(d+n)(p-2))}{C_{d,n, q}(q-2)((d+n)(q-2)-8)}\right)^{\frac{8}{(d+n)(q-p)}} \Theta^{\frac{d+n-4}{d+n}} .
$$
Then by a direct calculation, we have $g_1^{\prime \prime}(t) \leq 0$ if and only if $t \geq t_2$. Hence
$$
\max\limits_{t \in \mathbb{R}^{+}} g_1(t)=\max _{t \in\left[t_2, \infty\right)} g_1(t).
$$
Note that for any $t \geq t_2$,
\begin{eqnarray}\label{eq5.8}
g_1(t) & =&\frac{1-\left\|V_{-}\right\|_{\frac{d+n}{4}} S^{-1}}{2} t-\frac{C_{d,n, q} \Theta^{\frac{4 q-(d+n)(q-2)}{8}}}{q} t^{\frac{(d+n)(q-2)}{8}}-\mu C_{d,n, p} \Theta^{\frac{4 p-(d+n)(p-2)}{8}}t^{\frac{(d+n)(p-2)}{8}} \nonumber\\
& =&\frac{1}{2}\left(1-\left\|V_{-}\right\|_{\frac{d+n}{4}} S^{-1}\right) t-\mu C_{d,n, p} \Theta^{\frac{(q-p)(d+n-4)}{8}}\cdot\Theta^{\frac{4 q-(d+n)(q-2)}{8}}t^{\frac{(d+n)(p-2)}{8}}\nonumber\\
&&-\frac{C_{d,n, q} \Theta^{\frac{4 q-(d+n)(q-2)}{8}}}{q} t^{\frac{(d+n)(q-2)}{8}} \nonumber\\
& =&\frac{1}{2}\left(1-\left\|V_{-}\right\|_{\frac{d+n}{4}} S^{-1}\right) t-\frac{C_{d,n,q}(q-2)((d+n)(q-2)-8)}{q(p-2)(8-(d+n)(p-2))}t_2^{\frac{(d+n)(q-p)}{8}}\nonumber\\
&& \cdot\Theta^{\frac{4 q-(d+n)(q-2)}{8}}t^{\frac{(d+n)(p-2)}{8}}-\frac{C_{N, q} \Theta^{\frac{2 q-N(q-2)}{4}}}{q} t^{\frac{N(q-2)}{4}} \nonumber\\
&\geq&\frac{1-\left\|V_{-}\right\|_{\frac{d+n}{4}} S^{-1}}{2} t-\frac{C_{d,n,q}(q-2)((d+n)(q-2)-8)}{q(p-2)(8-(d+n)(p-2))}  \cdot\Theta^{\frac{4 q-(d+n)(q-2)}{8}}t^{\frac{(d+n)(q-2)}{8}}\nonumber\\
&&-\frac{C_{d,n, q} \Theta^{\frac{4 q-(d+n)(q-2)}{8}}}{q} t^{\frac{(d+n)(q-2)}{8}} \nonumber\\
&=& \frac{1-\left\|V_{-}\right\|_{\frac{d+n}{4}} S^{-1}}{2} t-\left(\frac{C_{d,n,q}(q-2)((d+n)(q-2)-8)}{q(p-2)(8-(d+n)(p-2))} +\frac{C_{d,n, q}}{q}\right) \Theta^{\frac{4 q-(d+n)(q-2)}{8}} \nonumber\\
&&\cdot t^{\frac{(d+n)(q-2)}{8}}\nonumber\\
& =:& g_2(t) .
\end{eqnarray}

Now, we will determine the value of $\widetilde{\Theta}_V$. In fact, $g_1\left(l_M\right)=\max\limits_{t \in \mathbb{R}^{+}} g_1(t)>0$ as long as $g_2(t_2)>0$, that is,
\begin{eqnarray*}
g_2(t_2)&=&\frac{1-\left\|V_{-}\right\|_{\frac{d+n}{4}} S^{-1}}{2} t_2-\left(\frac{C_{d,n,q}(q-2)((d+n)(q-2)-8)}{q(p-2)(8-(d+n)(p-2))} +\frac{C_{d,n, q}}{q}\right) \Theta^{\frac{4 q-(d+n)(q-2)}{8}}\\
&&\cdot t_2^{\frac{(d+n)(q-2)}{8}}\\
&=&\frac{1}{2}\left(1-\left\|V_{-}\right\|_{\frac{d+n}{4}} S^{-1}\right)\left(\frac{\mu q C_{d,n, p}}{C_{d,n, q}A_{p,q}}\right)^{\frac{8}{(d+n)(q-p)}} \Theta^{\frac{d+n-4}{d+n}}-\left(\frac{C_{d,n,q}}{q}A_{p,q}+\frac{C_{d,n, q}}{q}\right)\\
&&\cdot\Theta^{\frac{4 q-(d+n)(q-2)}{8}}\cdot\left[\left(\frac{\mu q C_{d,n, p}}{C_{d,n, q}A_{p,q}}\right)^{\frac{8}{(d+n)(q-p)}}\right]^{\frac{(d+n)(q-2)}{8}}\cdot\Theta^{\frac{d+n-4}{d+n}\cdot\frac{(d+n)(q-2)}{8}}\\
&=&\frac{1}{2}\left(1-\left\|V_{-}\right\|_{\frac{d+n}{4}} S^{-1}\right)\left(\frac{\mu q C_{d,n, p}}{C_{d,n, q}A_{p,q}}\right)^{\frac{8}{(d+n)(q-p)}} \Theta^{\frac{d+n-4}{d+n}}-\left(\frac{C_{d,n,q}}{q}A_{p,q}+\frac{C_{d,n, q}}{q}\right)\cdot\Theta\\
&&\cdot\left(\frac{\mu q C_{d,n, p}}{C_{d,n, q}A_{p,q}}\right)^{\frac{q-2}{q-p}}\\
&>&0,
\end{eqnarray*}
where
$$
A_{p, q}=\frac{(q-2)((d+n)(q-2)-8)}{(p-2)(8-(d+n)(p-2))}.
$$
Hence, we obtain that
\begin{equation*}
  \frac{1-\left\|V_{-}\right\|_{\frac{d+n}{4}} S^{-1}}{2} \left(\frac{\mu q C_{d,n, p}}{C_{d,n, q}A_{p,q}}\right)^{\frac{8}{(d+n)(q-p)}} \Theta^{\frac{-4}{d+n}}>\left(\frac{C_{d,n,q}}{q}A_{p,q}+\frac{C_{d,n, q}}{q}\right)\left(\frac{\mu q C_{d,n, p}}{C_{d,n, q}A_{p,q}}\right)^{\frac{q-2}{q-p}},
\end{equation*}
which implies
\begin{equation*}
   \left(\frac{1-\left\|V_{-}\right\|_{\frac{d+n}{4}} S^{-1}}{2}\right)^{\frac{d+n}{4}}\left(\frac{C_{d,n,q}}{q}A_{p,q}+\frac{C_{d,n, q}}{q}\right)^{-\frac{d+n}{4}}\left(\frac{\mu q C_{d,n, p}}{C_{d,n, q}A_{p,q}}\right)^{\frac{8-(d+n)(q-2)}{4(d+n)(q-p)}} >\Theta,
\end{equation*}
Hence, we take
\begin{equation*}
 \widetilde{\Theta}_V= \left(\frac{1-\left\|V_{-}\right\|_{\frac{d+n}{4}} S^{-1}}{2}\right)^{\frac{d+n}{4}}\left(\frac{C_{d,n,q}}{q}A_{p,q}+\frac{C_{d,n, q}}{q}\right)^{-\frac{d+n}{4}}\left(\frac{\mu q C_{d,n, p}}{C_{d,n, q}A_{p,q}}\right)^{\frac{8-(d+n)(q-2)}{4(d+n)(q-p)}}.
\end{equation*}
Let
$$
A=\left(\frac{C_{d,n,q}(q-2)((d+n)(q-2)-8)}{q(p-2)(8-(d+n)(p-2))} +\frac{C_{d,n, q}}{q}\right)
$$
and
$$
t_g=\left[\frac{4\left(1-\left\|V_{-}\right\|_{\frac{d+n}{4}} S^{-1}\right)}{(d+n)(q-2) A}\right]^{\frac{8}{(d+n)(q-2)-8}} \Theta^{\frac{(d+n)(q-2)-4 q}{(d+n)(q-2)-8}},
$$
so that $t_g>t_2$ by the definition of $\widetilde{\Theta}_V, \max\limits_{t \in\left[t_2, \infty\right)} g_2(t)=g_2(t_g)$ and
\begin{eqnarray*}
&&\max\limits_{t \in \mathbb{R}^{+}} g_1(t) \\
& \geq& \max\limits_{t \in\left[t_2, \infty\right)} g_2(t) \\
&=&\frac{1}{2}\left(1-\left\|V_{-}\right\|_{\frac{d+n}{4}} S^{-1}\right) t_g-A \Theta^{\frac{4 q-(d+n)(q-2)}{8}}t_g^{\frac{(d+n)(q-2)}{8}}\\
&=&\frac{1}{2}\left(1-\left\|V_{-}\right\|_{\frac{d+n}{4}} S^{-1}\right) \left[\frac{4\left(1-\left\|V_{-}\right\|_{\frac{d+n}{4}} S^{-1}\right)}{(d+n)(q-2) A}\right]^{\frac{8}{(d+n)(q-2)-8}} \Theta^{\frac{(d+n)(q-2)-4 q}{(d+n)(q-2)-8}}\\
&&-A \Theta^{\frac{4 q-(d+n)(q-2)}{8}}\left[\frac{4\left(1-\left\|V_{-}\right\|_{\frac{d+n}{4}} S^{-1}\right)}{(d+n)(q-2) A}\right]^{\frac{(d+n)(q-2)}{(d+n)(q-2)-8}} \Theta^{\frac{(d+n)(q-2)-4 q}{(d+n)(q-2)-8}\cdot\frac{(d+n)(q-2)}{8}}\\
& =&\frac{((d+n)(q-2)-8)}{8}\left[\frac{4\left(1-\left\|V_{-}\right\|_{\frac{d+n}{4}} S^{-1}\right)}{(d+n)(q-2)}\right]^{\frac{(d+n)(q-2)}{(d+n)(q-2)-8}} A^{\frac{8}{8-(d+n)(q-2)}} \Theta^{\frac{(d+n)(q-2)-4 q}{(d+n)(q-2)-8}} .
\end{eqnarray*}
Set $\bar{r}_\Theta=\max \left\{\frac{1}{t_1}, \sqrt{\frac{2 \theta \Theta}{t_g}}\right\}$, then $v_{\frac{1}{\bar{r}_\Theta}} \in S_{r, \Theta}$ for any $r>\bar{r}_\Theta$, and
\begin{equation}\label{eq5.9}
  \left\|\Delta v_{\frac{1}{\bar{r}_\Theta}}\right\|_2^2=\left(\frac{1}{\bar{r}_\Theta}\right)^4\left\|\Delta v_1\right\|_2^2<t_g=\left[\frac{4\left(1-\left\|V_{-}\right\|_{\frac{d+n}{4}} S^{-1}\right)}{(d+n)(q-2) A}\right]^{\frac{8}{(d+n)(q-2)-8}} \Theta^{\frac{(d+n)(q-2)-4 q}{(d+n)(q-2)-8}}.
\end{equation}
Moreover,
\begin{equation}\label{eq5.10}
J_{\bar{r}_\Theta, s}\left(v_{\frac{1}{\bar{r}_\Theta}}\right) \leq h_2\left(\frac{1}{\bar{r}_\Theta}\right) \leq h_2\left(t_1\right) .
\end{equation}
Let $u^0=v_{\frac{1}{\bar{\tau}_\Theta}}, u^1=v_{t_0}$ and
$$
\widetilde{r}_\Theta=\max \left\{\frac{1}{t_0}, \bar{r}_\Theta\right\} .
$$
Then the statement (i) holds by \eqref{eq5.5}, \eqref{eq5.6}, \eqref{eq5.9}, \eqref{eq5.10}.

(ii) holds by \eqref{eq5.8} and a direct calculation.

(iii) In view of $J_{r, s}\left(u^1\right) \leq 0$ for any $\gamma \in \Gamma_{r, \Theta}$ and the definition of $t_0$, we have
$$
\|\Delta \gamma(0)\|_2^2<t_g<\|\Delta \gamma(1)\|_2^2 .
$$
It then follows from \eqref{eq5.8} that
\begin{eqnarray*}
&&\max\limits_{t \in[0,1]} J_{r, s}(\gamma(t)) \\
& \geq& g_2\left(t_g\right) \\
& =&\frac{((d+n)(q-2)-8)}{8}\left[\frac{4\left(1-\left\|V_{-}\right\|_{\frac{d+n}{4}} S^{-1}\right)}{(d+n)(q-2)}\right]^{\frac{(d+n)(q-2)}{(d+n)(q-2)-8}} A^{\frac{8}{8-(d+n)(q-2)}} \Theta^{\frac{(d+n)(q-2)-4 q}{(d+n)(q-2)-8}}
\end{eqnarray*}
for any $\gamma \in \Gamma_{r, \Theta}$, hence the first inequality in (iii) holds. We define a path $\gamma:[0,1] \rightarrow S_{r, \Theta}$ by $\gamma(t): \Omega_r\times\mathbb{T}^n \rightarrow \mathbb{R}$,
$$
 (x,y) \mapsto\left(\tau t_0+(1-\tau) \frac{1}{\widetilde{r}}_\Theta\right)^{\frac{d+n}{2}} v_1\left(\left(\tau t_0+(1-\tau) \frac{1}{\widetilde{r}_\Theta}\right) x,\left(\tau t_0+(1-\tau) \frac{1}{\widetilde{r}_\Theta}\right) y\right) .
$$
Then $\gamma \in \Gamma_{r, \Theta}$, and the second inequality in (iii) follows from \eqref{eq5.4}.
\end{proof}
\begin{lemma}\label{L5.2}
Assume $0<\Theta<\widetilde{\Theta}_V$ where $\widetilde{\Theta}_V$ is given in Theorem \ref{t1.3}. Let $r>\widetilde{r}_\Theta$, where $\widetilde{r}_\Theta$ is defined in Lemma \ref{L5.1}. Then problem \eqref{eq5.1} admits a solution $\left(\lambda_{r, s}, u_{r, s}\right)$ for almost every $s \in\left[\frac{1}{2}, 1\right]$ and $J_{r, s}\left(u_{r, s}\right)=m_{r, s}(\Theta)$.
\end{lemma}
\begin{proof}
The proof is similar to the Lemma \ref{L3.2}. We omit it here.
\end{proof}

\begin{lemma}\label{L5.3}
For fixed $\Theta>0$ the set of solutions $u \in S_{r, \Theta}$ of \eqref{eq5.1} is bounded uniformly in $s$ and $r$.
\end{lemma}
\begin{proof}
Since $u$ is a solution of \eqref{eq5.1}, we have
\begin{equation*}
  \int_{\Omega_r\times\mathbb{T}^n}|\Delta u|^2 d xdy+\int_{\Omega_r\times\mathbb{T}^n} V u^2 d xdy=s \int_{\Omega_r\times\mathbb{T}^n}|u|^q d xdy+s\mu \int_{\Omega_r\times\mathbb{T}^n}|u|^p d x-\lambda \int_{\Omega_r\times\mathbb{T}^n}|u|^2 d x .
\end{equation*}
The Pohozaev identity implies
\begin{eqnarray*}
&&\frac{d+n-4}{2(d+n)} \int_{\Omega_r\times\mathbb{T}^n}|\Delta u|^2 d x+\frac{1}{2(d+n)} \int_{\partial (\Omega_r\times\mathbb{T}^n)}|\Delta u|^2((x,y) \cdot \mathbf{n}) d \sigma\\
&&+\frac{1}{2(d+n)} \int_{\Omega_r\times\mathbb{T}^n}\widetilde{V} u^2dxdy+\frac{1}{2} \int_{\Omega_r\times\mathbb{T}^n} V u^2 d xdy \\
&=&-\frac{\lambda}{2} \int_{\Omega_r\times\mathbb{T}^n}|u|^2 d x+\frac{s}{q} \int_{\Omega_r\times\mathbb{T}^n}|u|^q d x+\frac{s\mu}{p} \int_{\Omega_r\times\mathbb{T}^n}|u|^p d xdy.
\end{eqnarray*}
It then follows from $\mu >0$ that
\begin{eqnarray*}
&&\frac{2}{d+n} \int_{\Omega_r\times\mathbb{T}^n}|\Delta u|^2 d  xdy-\frac{1}{2(d+n)} \int_{\partial (\Omega_r\times\mathbb{T}^n)}|\Delta u|^2((x,y) \cdot \mathbf{n}) d \sigma\\
&&-\frac{1}{2(d+n)} \int_{\Omega_r\times\mathbb{T}^n}(\nabla V \cdot (x,y)) u^2 d xdy \\
& =&\frac{(q-2) s}{2 q} \int_{\Omega_r\times\mathbb{T}^n}|u|^q d xdy+ s\mu\left(\frac{1}{2}-\frac{1}{p}\right)\int_{\Omega_r\times\mathbb{T}^n}|u|^pd xdy \\
& \geq&\frac{q-2}{2}\left(\frac{1}{2} \int_{\Omega_r\times\mathbb{T}^n}|\Delta u|^2 d xdy+\frac{1}{2} \int_{\Omega_r\times\mathbb{T}^n} V u^2 d xdy-m_{r, s}(\Theta)\right)\\
&&+  s\frac{\mu (p-q) }{2  }  \int_{\Omega_r\times\mathbb{T}^n}|u|^p d xdy .
\end{eqnarray*}
Using Gagliardo-Nirenberg inequality and (iii) in Lemma \ref{L5.1}, we have
 \begin{eqnarray*}
&&\frac{q-2}{2} m_{r, s}(\Theta)\\
&\geq & \frac{(d+n)(p-2)-8}{4(d+n)} \int_{\Omega_r\times\mathbb{T}^n}|\Delta u|^2 d x-\Theta\left(\frac{1}{2 (d+n)}\|\nabla V \cdot (x,y)\|_{\infty}+\frac{p-2}{4}\|V\|_{\infty}\right)\\
&&+\frac{s\mu (p-q) }{2  }C_{d,n, p} \Theta^{\frac{2 p-(d+n)(p-2)}{4}}\left(\int_{\Omega_r\times\mathbb{T}^n}|\Delta u|^2 d xdy\right)^{\frac{(d+n)(p-2)}{4}}.
 \end{eqnarray*}
Since $2<p<2+\frac{8}{d+n}$, we can bound $\int_{\Omega_r\times\mathbb{T}^n}|\Delta u|^2 d xdy$ uniformly in $s$ and $r$.
\end{proof}

\begin{lemma}\label{L5.4}
Assume $0<\Theta<\widetilde{\Theta}_V$, where $\widetilde{\Theta}_V$ is given in Theorem \ref{t1.3}, and let $r>\widetilde{r}_\Theta$, where $\widetilde{r}_\Theta$ is defined in Lemma \ref{L5.1}. Then the following hold:

(i) Equation \eqref{eq1.1} admits a solution $\left(\lambda_{r, \Theta}, u_{r, \Theta}\right)$ for every $r>\widetilde{r}_\Theta$.

(ii) There is $0<\bar{\Theta} \leq \widetilde{\Theta}_V$ such that
$$
\liminf\limits_{r \rightarrow \infty} \lambda_{r, \Theta}>0 \ \text {for any}\ 0<\Theta<\bar{\Theta} .
$$
\end{lemma}
\begin{proof}
The proof of $(i)$ is similar to that of Lemma \ref{L3.4}, we omit it. As be consider $H_0^2\left(\Omega_r\times\mathbb{T}^n\right)$ as a subspace of $H^2(\mathbb{R}^d\times\mathbb{T}^n)$ for every $r>0$. In view of Lemma \ref{L5.3}, there are $\lambda_\Theta$ and $u_\Theta \in H^2(\mathbb{R}^d\times\mathbb{T}^n)$ such that, up to a subsequence,
$$
u_{r, \Theta} \rightharpoonup u_\Theta \ \text {in}\ H^2(\mathbb{R}^d\times\mathbb{T}^n)\ \text {and}\  \lim _{r \rightarrow \infty} \lambda_{r, \Theta} \rightarrow \lambda_\Theta .
$$
Arguing by contradiction, we assume that $\lambda_{\Theta_n} \leq 0$ for some sequence $\Theta_n \rightarrow 0$. Let $\theta_r$ be the principal eigenvalue of $\Delta^2$ with Dirichlet boundary condition in $\Omega_r\times\mathbb{T}^n$ and let $v_r>0$ be the corresponding normalized eigenfunction. Testing \eqref{eq1.1} with $v_r$, it holds
$$
\left(\theta_r+\lambda_{r, \Theta_n}\right) \int_{\Omega_r\times\mathbb{T}^n} u_{r, \Theta_n} v_r d xdy+\int_{\Omega_r\times\mathbb{T}^n} V u_{r, \Theta_n} v_r d xdy \geq 0 .
$$
In view of $\int_{\Omega_r\times\mathbb{T}^n} u_{r,\Theta_n} v_r d xdy>0$ and $\theta_r=r^{-4} \theta_1$, there holds
$$
\max\limits_{(x,y) \in \mathbb{R}^d\times\mathbb{T}^n} V+\lambda_{r,\Theta_n}+r^{-4} \theta_1 \geq 0 .
$$
Hence there exists $C>0$ independent of $n$ such that $\left|\lambda_{\Theta_n}\right| \leq C$ for any $n$.

\textbf{Case 1} There is subsequence denoted still by $\left\{\Theta_n\right\}$ such that $u_{\Theta_n}=0$. We first claim that there exists $d_n>0$ for any $n$ such that
\begin{equation}\label{eq5.11}
  \liminf\limits_{r \rightarrow \infty} \sup\limits_{z \in \mathbb{R}^d} \int_{B(z, 1)\times\mathbb{T}^n} u_{r, \Theta_n}^2 d xdy \geq d_n .
\end{equation}
Otherwise, the concentration compactness principle implies for every $n$ that
$$
u_{r, \Theta_n} \rightarrow 0 \text { in } L^t(\mathbb{R}^d\times\mathbb{T}^n) \quad \text { as } r \rightarrow \infty, \quad \text { for all } 2<t<4^* .
$$
By the diagonal principle, \eqref{eq1.1} and $\left|\lambda_{r, \Theta_n}\right| \leq 2 C$ for large $r$, there exists $r_n \rightarrow \infty$ such that
$$
\int_{\Omega_r\times\mathbb{T}^n}\left|\Delta u_{r_n, \Theta_n}\right|^2 d xdy \leq C
$$
for some $C$ independent of $n$, contradicting (iii) in Lemma \ref{L5.1} for large $n$. As a consequence \eqref{eq5.11} holds, and there is $z_{r, \Theta_n} \in \Omega_r\times\mathbb{T}^n$ with $\left|z_{r, \Theta_n}\right| \rightarrow \infty$ such that
$$
\int_{B\left(z_{r, \Theta_n}, 1\right)} u_{r, \Theta_n}^2 d xdy \geq \frac{d_n}{2} .
$$
Moreover, $\operatorname{dist}\left(z_{r, \Theta_n}, \partial (\Omega_r\times\mathbb{T}^n)\right) \rightarrow \infty$ as $r \rightarrow \infty$ by an argument similar to that in Lemma \ref{L3.6}. Now, for $n$ fixed let $v_r(x)=u_{r, \Theta_n}\left(x+z_{r, \Theta_n}\right)$ for
$$
x \in \Sigma^r:=\left\{x \in \mathbb{R}^d\times\mathbb{T}^n: x+z_{r, \Theta_n} \in \Omega_r\times\mathbb{T}^n\right\}.
$$
It follows from Lemma \ref{L5.3} that there is $v \in H^2(\mathbb{R}^d\times\mathbb{T}^n)$ with $v \neq 0$ such that $v_r \rightharpoonup v$. Observe that for every $\phi \in C_c^{\infty}(\mathbb{R}^d\times\mathbb{T}^n)$ there is $r$ large such that $\phi(\cdot-z_{r, \Theta_n}) \in C_c^{\infty}\left(\Omega_r\times\mathbb{T}^n\right)$ due to $\operatorname{dist}\left(z_{r, \Theta_n}, \partial (\Omega_r\times\mathbb{T}^n)\right) \rightarrow \infty$ as $r \rightarrow \infty$. It follows that
\begin{eqnarray}\label{eq5.12}
&&\int_{\Omega_r\times\mathbb{T}^n} \Delta u_{r, \Theta_n} \Delta \phi\left(\cdot-z_{r, \Theta_n}\right) d xdy+\int_{\Omega_r\times\mathbb{T}^n} V u_{r, \Theta_n} \phi\left(\cdot-z_{r, \Theta_n}\right) d xdy \nonumber\\
&=&\int_{\Omega_r\times\mathbb{T}^n}\left|u_{r, \Theta_n}\right|^{q-2} u_{r, \Theta_n} \phi\left(\cdot-z_{r, \Theta_n}\right) d xdy+\mu \int_{\Omega_r\times\mathbb{T}^n}\left|u_{r, \Theta_n}\right|^{p-2} u_{r, \Theta_n} \phi\left(\cdot-z_{r, \Theta_n}\right) d xdy\nonumber\\
&&-\lambda_{r, \Theta_n} \int_{\Omega_r\times\mathbb{T}^n} u_{r, \Theta_n} \phi\left(\cdot-z_{r, \Theta_n}\right) d xdy.
\end{eqnarray}
Using $|z_{r,\Theta_n}|\rightarrow\infty$ as $r\rightarrow\infty$, it follows that
\begin{eqnarray*}
\left|\int_{\Omega_r\times\mathbb{T}^n} V u_{r, \Theta_n} \phi\left(\cdot-z_{r, \Theta_n}\right) d xdy\right| & \leq& \int_{S u p p \phi}\left|V\left(\cdot+z_{r, \Theta_n}\right) v_r \phi\right| d xdy \\
& \leq&\left\|v_r\right\|_{4^*}\|\phi\|_{4^*}\left(\int_{\mathbb{R}^d \backslash B_{\frac{|z_{r, \Theta_n }|}{2}}\times\mathbb{T}^n}|V|^{\frac{d+n}{4}} d xdy\right)^{\frac{4}{d+n}}  \\
& \rightarrow& 0 \quad \text { as } r \rightarrow \infty.
\end{eqnarray*}
Letting $r \rightarrow \infty$ in \eqref{eq5.12}, we get for every $\phi \in C_c^{\infty}(\mathbb{R}^d\times\mathbb{T}^n)$ :
$$
\int_{\mathbb{R}^d\times\mathbb{T}^n} \Delta v \cdot \Delta \phi d xdy+\lambda_{\Theta_n} \int_{\mathbb{R}^d\times\mathbb{T}^n} v \phi d xdy=\int_{\mathbb{R}^d\times\mathbb{T}^n}|v|^{q-2} v \phi d x+\mu \int_{\mathbb{R}^d\times\mathbb{T}^n}|v|^{p-2} v \phi d xdy.
$$
Therefore $v \in H^2(\mathbb{R}^d\times\mathbb{T}^n)$ is a weak solution of the equation
$$
\Delta^2 v+\lambda_{\Theta_n} v=\mu |v|^{p-2} v+|v|^{q-2} v \quad \text { in } \mathbb{R}^d\times\mathbb{T}^n
$$
and
$$
\int_{\mathbb{R}^d\times\mathbb{T}^n}|\Delta v|^2 d xdy+\lambda_{\Theta_n} \int_{\mathbb{R}^d\times\mathbb{T}^n}|v|^2 d xdy=\mu \int_{\mathbb{R}^d\times\mathbb{T}^n}|v|^p d xdy+\int_{\mathbb{R}^d\times\mathbb{T}^n}|v|^q d xdy.
$$
The Pohozaev identity implies
$$
\frac{d+n-4}{2(d+n)} \int_{\mathbb{R}^d\times\mathbb{T}^n}|\Delta v|^2 d xdy+\frac{\lambda_{\Theta_n}}{2} \int_{\mathbb{R}^d\times\mathbb{T}^n}|v|^2 d xdy=\frac{\mu}{p} \int_{\mathbb{R}^d\times\mathbb{T}^n}|v|^p d xdy+\frac{1}{q} \int_{\mathbb{R}^d\times\mathbb{T}^n}|v|^q d xdy,
$$
hence
\begin{eqnarray}\label{eq5.13}
&&\frac{\mu(2(d+n)-p(d+n-4))}{2p(d+n)}\int_{\mathbb{R}^d\times\mathbb{T}^n}|v|^pd xdy+\frac{2(d+n)-q(d+n-4)}{2q(d+n)} \int_{\mathbb{R}^\times\mathbb{T}^n}|v|^q d xdy \nonumber\\
&=&\frac{2\lambda_{\Theta_n}}{d+n} \int_{\mathbb{R}^d\times\mathbb{T}^n}|v|^2 d xdy.
\end{eqnarray}
We have $\lambda_{\Theta_n}>0$ because of $2<p<2+\frac{8}{d+n}<q<4^*$, which is a contradiction.

\textbf{Case 2} $u_{\Theta_n} \neq 0$ for $n$ large. Note that $u_{\Theta_n}$ satisfies
\begin{equation}\label{eq5.14}
  \Delta^2 u_{\Theta_n}+V u_{\Theta_n}+\lambda_{\Theta_n} u_{\Theta_n}=\mu \left|u_{\Theta_n}\right|^{p-2} u_{\Theta_n}+\left|u_{\Theta_n}\right|^{q-2} u_{\Theta_n} .
\end{equation}
If $v_{r, \Theta_n}:=u_{r, \Theta_n}-u_{\Theta_n}$ satisfies
\begin{equation}\label{eq5.15}
  \limsup\limits_{r \rightarrow \infty} \max\limits_{z \in \mathbb{R}^d\times\mathbb{T}^n} \int_{B(z, 1)} v_{r, \Theta_n}^2 d xdy=0,
\end{equation}
then the concentration compactness principle implies $u_{r, \Theta_n} \rightarrow u_{\Theta_n}$ in $L^t(\mathbb{R}^d\times\mathbb{T}^n)$ for any $2<t<4^*$. It then follows from \eqref{eq1.1} and \eqref{eq5.14} that
\begin{eqnarray*}
&&\int_{\Omega_r\times\mathbb{T}^n}\left|\Delta u_{r, \Theta_n}\right|^2 d xdy+\Theta_n \lambda_{r, \Theta_n} \\
& =&\mu \int_{\Omega_r\times\mathbb{T}^n}|u_{r, \Theta_n}|^p d xdy+\int_{\Omega_r\times\mathbb{T}^n}\left|u_{r, \Theta_n}\right|^q d xdy-\int_{\Omega_r\times\mathbb{T}^n} V u_{r, \Theta_n}^2 d xdy \\
& \rightarrow& \mu \int_{\mathbb{R}^d\times\mathbb{T}^n}|u_{\Theta_n}|^p d xdy+\int_{\mathbb{R}^d\times\mathbb{T}^n}\left|u_{\Theta_n}\right|^q d xdy-\int_{\mathbb{R}^d\times\mathbb{T}^n} V u_{\Theta_n}^2 d xdy \\
& =&\int_{\mathbb{R}^d\times\mathbb{T}^n}\left|\Delta u_{\Theta_n}\right|^2 d xdy+\lambda_{\Theta_n} \int_{\mathbb{R}^d\times\mathbb{T}^n} u_{\Theta_n}^2 d xdy.
\end{eqnarray*}
Using $\lambda_{r, \Theta_n} \rightarrow \lambda_{\Theta_n}$ as $r \rightarrow \infty$, we further have
\begin{equation}\label{eq5.16}
\int_{\Omega_r\times\mathbb{T}^n}\left|\Delta u_{r, \Theta_n}\right|^2 d xdy+\Theta_n \lambda_{\Theta_n} \rightarrow \int_{\mathbb{R}^d\times\mathbb{T}^n}\left|\Delta u_{\Theta_n}\right|^2 d xdy+\lambda_{\Theta_n} \int_{\mathbb{R}^d\times\mathbb{T}^n} u_{\Theta_n}^2 d xdy
\end{equation}
as $r \rightarrow \infty$. Using \eqref{eq5.16}, (iii) in Lemma \ref{L5.1} and $\left|\lambda_{\Theta_n}\right| \leq C$ for large $n$, there holds
$$
\int_{\mathbb{R}^d\times\mathbb{T}^n}\left|\Delta u_{\Theta_n}\right|^2 d xdy \rightarrow \infty \quad \text { as } n \rightarrow \infty .
$$
By \eqref{eq5.14} and the Pohozaev identity
\begin{eqnarray*}
& &\frac{d+n-4}{2(d+n)} \int_{\mathbb{R}^d\times\mathbb{T}^n}\left|\Delta u_{\Theta_n}\right|^2 d xdy+\frac{1}{2(d+n)} \int_{\mathbb{R}^d\times\mathbb{T}^n} \widetilde{V} u_{\Theta_n}^2dxdy+\frac{1}{2} \int_{\mathbb{R}^d\times\mathbb{T}^n} V  u_{\Theta_n}^2 d xdy \nonumber\\
& =&\frac{1}{q} \int_{\mathbb{R}^d\times\mathbb{T}^n}\left|u_{\Theta_n}\right|^q d x+\frac{\mu}{p} \int_{\mathbb{R}^d\times\mathbb{T}^n}|u_{\Theta_n}|^p d xdy-\frac{\lambda_\Theta}{2} \int_{\mathbb{R}^d\times\mathbb{T}^n} u_{\Theta_n}^2 d xdy,
\end{eqnarray*}
it holds that
\begin{eqnarray*}
0 &\leq & \frac{(2-q) \lambda_{\Theta_n}}{2 q} \int_{\mathbb{R}^d\times\mathbb{T}^n} u_{\Theta_n}^2 d xdy \\
&\leq & \frac{(d+n-4) q-2(d+n)}{2q(d+n)} \int_{\mathbb{R}^d\times\mathbb{T}^n}\left|\Delta u_{\Theta_n}\right|^2 d xdy+\frac{\|\widetilde{V}\|_{\infty}}{2(d+n)} \Theta_n+\frac{(q-2)\|V\|_{\infty}}{2 q} \Theta_n \\
&\rightarrow & -\infty  \text { as } n \rightarrow \infty .
\end{eqnarray*}
Therefore \eqref{eq5.15}  cannot occur. Consequently there exist $d_n>0$ and $z_{r, \Theta_n} \in \Omega_r\times\mathbb{T}^n$ with $\left|z_{r, \Theta_n}\right| \rightarrow \infty$ as $r \rightarrow \infty$ such that
$$
\int_{B(z_{r, \Theta_n}, 1)} v_{r, \Theta_n}^2 d xdy>d_n.
$$
Then $\widetilde{v}_{r, \Theta_n}:=v_{r, \Theta_n}\left(\cdot+z_{r, \Theta_n}\right) \rightharpoonup \widetilde{v}_{\Theta_n} \neq 0$, and $\widetilde{v}_{\Theta_n}$ is a nonnegative solution of
$$
\Delta^2 v+\lambda_{\Theta_n} v=\beta f(v) v+|v|^{q-2} v \ \text { in } \mathbb{R}^d\times\mathbb{T}^n .
$$
In fact, we have $\liminf\limits_{r \rightarrow \infty} \operatorname{dist}\left(z_{r, \Theta_n}, \partial (\Omega_r\times\mathbb{T}^n)\right)=\infty$ by the Liouville theorem on the half space. It follows from an argument similar to that of \eqref{eq5.13} that $\lambda_{\Theta_n}>0$ for large $n$, which is a contradiction.
\end{proof}

\noindent\textbf{Proof of Theorem \ref{t1.3}} The proof is a direct consequence of Lemma \ref{L5.4} and  Lemma \ref{L3.5}.

\section{Orbital stability}
In this section,  we will study the orbital stability of the solution obtained in Theorems \ref{t1.2} and \ref{t1.4}.   Firstly, we give the definition of orbital stability.
\begin{definition}
A set $\mathcal{M} \subset H^2(\mathbb{R}^d\times\mathbb{T}^n)$ is orbitally stable under the flow associated with the problem
\begin{equation}\label{eq6.1}
   \left\{\begin{array}{l}i \frac{\partial \psi}{\partial t}-\Delta^2 \psi-V(x,y)\psi+\mu|\psi|^{p-2}\psi+|\psi|^{q-2}\psi =0,\ t>0,\ x \in \mathbb{R}^d\times\mathbb{T}^n, \\ \psi(0, x,y)=u_0(x,y),\end{array}\right.
\end{equation}
if $\forall \theta>0,$ there exists $\gamma>0$ such that for any $u_0 \in H^2(\mathbb{R}^d\times\mathbb{T}^n)$ satisfying $\mathop{\operatorname{dist}}\limits_{H^2(\mathbb{R}^d\times\mathbb{T}^n)}(u_0, \mathcal{M})<\gamma$,
the solution $\psi(t, \cdot, \cdot)$ of problem \eqref{eq6.1} with $\psi(0, x, y)=u_0$ satisfies
$$
\sup _{t \in \mathbb{R}^{+}} \mathop{\operatorname{dist}}\limits_{H^2(\mathbb{R}^d\times\mathbb{T}^n)}(\psi(t, \cdot, \cdot), \mathcal{M})<\theta .
$$
\end{definition}

\begin{proof}[\bf Proof of Theorem \ref{t1.7}]
Similar to \eqref{eq3.3}, we have
\begin{eqnarray*}
I(u)&=&\frac{1}{2} \int_{\mathbb{R}^d\times\mathbb{T}^n}[|\Delta_{x,y}u|^2+V(x,y)u^2] d xdy -\frac{1}{q} \int_{\mathbb{R}^d\times\mathbb{T}^n}|u|^q d xdy-\frac{\mu}{p} \int_{\mathbb{R}^d\times\mathbb{T}^n}|u|^p d xdy
\\
&\geq&\frac{1-\left\|V_{-}\right\|_{\frac{d+n}{4}} S^{-1}}{2} \int_{\mathbb{R}^d\times\mathbb{T}^n}|\Delta u|^2 d xdy-\frac{C_{d,n, q} \Theta^{\frac{4q-(d+n)(q-2)}{8}}}{q}\left(\int_{\mathbb{R}^d\times\mathbb{T}^n}|\Delta u|^2 d xdy\right)^{\frac{(d+n)(q-2)}{8}}\\
&&-\mu C_{d,n, p} \Theta^{\frac{4 p-(d+n)(p-2)}{8}}\left(\int_{\mathbb{R}^d\times\mathbb{T}^n}|\Delta u|^2 d xdy\right)^{\frac{(d+n)(p-2)}{8}}\\
&=&Q_\Theta(\|\Delta u\|_2),
\end{eqnarray*}
where $Q_\Theta(m)=\frac{1-\left\|V_{-}\right\|_{\frac{d+n}{4}} S^{-1}}{2}m^2-\frac{C_{d,n, q} \Theta^{\frac{4q-(d+n)(q-2)}{8}}}{q}m^{\frac{(d+n)(q-2)}{4}}-\mu C_{d,n, p} \Theta^{\frac{4 p-(d+n)(p-2)}{8}}m^{\frac{(d+n)(p-2)}{4}}$. Define $\mathcal{M}$ as follows:
$$
\begin{aligned}
\mathcal{M}: &=\left\{v \in S_\Theta:\left.I\right|_{S_\Theta} ^{\prime}(v)=0, I(v)=e_{r, \Theta}\right\} \\
&=\left\{v \in S_\Theta:\left.I\right|_{S_\Theta} ^{\prime}(v)=0, I(v)=e_{r, \Theta},\ \|\Delta v\|_2 \leq R_1,\ Q_\Theta(R_1)=0\right\} .
\end{aligned}
$$
Now, we prove the stability of the sets $\mathcal{M}$. Denoting by $\psi(t, \cdot, \cdot)$ the solution to \eqref{eq6.1} with initial data $u_0$ and denoting by $\left[0, T_{\max }\right)$ the maximal existence interval for $\psi(t, \cdot, \cdot)$, we assume that $\psi(t, \cdot, \cdot)$ is globally defined for positive times. Next we prove that $\mathcal{M}$ is orbitally stable. We assume that the conclusion is invalid,  Then there exist $(\psi_0^j)_j \subset H^2(\mathbb{R}^d\times\mathbb{T}^n)$ and $(t_j)_j \subset \mathbb{R}^{+}$ such that
\begin{equation}\label{eq6.2}
  \lim\limits_{j \rightarrow \infty} \mathop{\operatorname{dist}}\limits_{H^2(\mathbb{R}^d\times\mathbb{T}^n)}\left(\psi_0^j, \mathcal{M}\right)=0,
\end{equation}
\begin{equation}\label{eq6.3}
  \liminf\limits_{j \rightarrow \infty} \mathop{\operatorname{dist}}\limits_{H^2(\mathbb{R}^d\times\mathbb{T}^n)}\left(\psi^j\left(t_j\right), \mathcal{M}\right)>0,
\end{equation}
where $\psi_j$ is the global solution of \eqref{eq6.1} with $\psi_j(0)=\psi_0^j$. By \eqref{eq6.2} we have
$$
\|\psi_0^j\|_2^2=\Theta+o_j(1),\ I(\psi_0^j)=e_\Theta+o_j(1).
$$
By conservation of mass and energy we infer that
$$
\|\psi_j\|_2^2=\Theta+o_j(1),\ I(\psi_j)=e_\Theta+o_j(1).
$$
By fundamental perturbation arguments, for $\tilde{\psi}_j:=\sqrt{\Theta} \|\psi_j\|_2^{-1} \psi_j \in S_\Theta$ we have
$$
\|\tilde{\psi}_j\|_2^2=\Theta+o_j(1),\ I(\tilde{\psi}_j)=e_\Theta+o_j(1).
$$
This implies that the sequence $(\tilde{\psi}_j)_j$ is a minimizing sequence of $e_\Theta$. From the proof of Theorem \ref{t1.2} we know that up to a subsequence, $\tilde{\psi}_j$ converges strongly to a minimizer $u_\Theta$ of $e_\Theta$ in $H^2(\mathbb{R}^d\times\mathbb{T}^n)$, which contradicts \eqref{eq6.3}. Hence $\mathcal{M}$ is orbitally stable.
\end{proof}

\end{document}